\documentclass[10pt]{amsart}
\usepackage[utf8]{inputenc}
\usepackage{colonequals}
\usepackage[T1]{fontenc}
\usepackage{textcomp} 
\usepackage{mathtools}
\usepackage{array}
\usepackage{amsmath,amssymb,amsthm, graphics, comment,bm}
\usepackage{diagmac2}
\usepackage{textcomp}
\usepackage{color}
\usepackage{xcolor}
\usepackage[alphabetic]{amsrefs}
\usepackage[all,cmtip,color,matrix,arrow]{xy}
\usepackage{yfonts}
\usepackage{tikz}
\usepackage{tikz-cd}
\usepackage{stmaryrd}
\usepackage{mathrsfs}
\usepackage{enumerate}
\usepackage{enumitem}
\usepackage{hyperref}
\hypersetup{
	colorlinks=true,
	linkcolor=blue,
	filecolor=magenta,      
	urlcolor=cyan,
}
\usepackage[all,cmtip]{xy}

\calclayout

\newcommand\spec{\operatorname{Spec}}

\newcommand\G{\mathbb{G}}
\newcommand\Ga{\mathbb{G}_a}
\newcommand\Gm{\mathbb{G}_m}

\newcommand\Proj{\text{Proj}}
\newcommand{\ra}{\rightarrow}

\newcommand\A{\mathbb{A}}
\newcommand\Z{\mathbb{Z}}

\newcommand\tX{\widetilde{X}}

\newcommand\tY{\widetilde{Y}}

\newcommand{\bA}{\mathbb{A}}

\newcommand{\bL}{\mathbb{L}}
\newcommand{\bN}{\mathbb{N}}
\newcommand{\bP}{\mathbb{P}}

\newcommand{\bT}{\mathbb{T}}
\newcommand{\bZ}{\mathbb{Z}}

\newcommand{\sC}{\mathscr{C}}

\newcommand{\sE}{\mathscr{E}}
\newcommand{\sF}{\mathscr{F}}

\newcommand{\oH}{\operatorname{H}}
\newcommand{\sN}{\mathscr{N}}
\newcommand{\sO}{\mathscr{O}}
\newcommand{\sP}{\mathscr{P}}
\newcommand{\sQ}{\mathscr{Q}}

\newcommand{\cC}{\mathcal{C}}
\newcommand{\cD}{\mathcal{D}}
\newcommand{\cE}{\mathcal{E}}

\newcommand{\cM}{\mathcal{M}}

\newcommand{\cO}{\mathcal{O}}
\newcommand{\cP}{\mathcal{P}}

\newcommand{\cX}{\mathcal{X}}
\newcommand{\cY}{\mathcal{Y}}
\newcommand{\cZ}{\mathcal{Z}}

\newcommand{\tcY}{\widetilde{\mathcal{Y}}}
\newcommand{\tcX}{\widetilde{\mathcal{X}}}

\newcommand{\Sym}{\operatorname{Sym}}

\newcommand{\Gr}{\operatorname{Gr}}

\newcommand{\GL}{\operatorname{GL}}

\newcommand{\im}{\operatorname{im}}
\newcommand{\Id}{\operatorname{Id}}

\newcommand{\Kring}{K^{\circ}}
\newcommand{\Kgroup}{K_{\circ}}
\newcommand{\opK}{\mathrm{op}K^{\circ}}

\usepackage{cleveref}

\newtheorem{theorem}{Theorem}[section]
\newtheorem{Thm}[theorem]{Theorem}
\newtheorem*{Thm*}{Theorem}
\newtheorem{Lemma}[theorem]{Lemma}
\newtheorem{Corollary}[theorem]{Corollary}

\newtheorem{Cor}[theorem]{Corollary}

\newtheorem{Prop}[theorem]{Proposition}
\newtheorem{notation}[theorem]{Notation}
\newtheorem*{Ques*}{Question}

\theoremstyle{definition}

\newtheorem*{Oss'}{Remark}

\newtheorem{Def}[theorem]{Definition}
\newtheorem*{Def*}{Definition}

\newtheorem{Rmk}[theorem]{Remark}
\newtheorem{Example}[theorem]{Example}

\definecolor{M}{RGB}{255,215,0}

\definecolor{V}{RGB}{255,95,144}

\definecolor{AC}{RGB}{255,0,0}

\definecolor{AL}{RGB}{10,200,100}

\title{K-theory of weighted blowups}

\author{Veronica Arena}
\address{Department of Pure Mathematics {\it \&} Mathematical Statistics, 
University of Cambridge, Cambridge, UK}
\email{va365@cam.ac.uk}

\author{Alessio Cela}
\address{Department of Pure Mathematics {\it \&} Mathematical Statistics, 
University of Cambridge, Cambridge, UK}
\email{ac2758@cam.ac.uk}

\author{Alberto Landi}
\address{Brown University, 151 Thayer Street, Providence, RI 02912, USA}
\email{alberto\_landi@brown.edu}

\author{Michele Pernice}
\address{University of Washington, 4110 E Stevens Way NE, Seattle, WA 98195, Stati Uniti}
\email{mpernice@uw.edu}
\begin{document}

\begin{abstract}
    We compute the K-theory of weighted blowups of smooth stacks satisfying the resolution property along smooth centers. As an application, we determine the K-theory of the stack of stable genus 1 curves with 2 marked points. Furthermore, we express the Lambda polynomial of the tangent complex of the blowup morphism in terms of the blowup data.
    
    Along the way, we generalize the Anderson–Payne construction of equivariant operational K-theory from torus actions to actions of smooth affine algebraic groups.
\end{abstract}

\maketitle


\section{Introduction}

Blowups are one of the most important constructions in the realm of algebraic geometry: they allow us to replace any given closed subset (called center) of an ambient space with a Cartier divisor (called exceptional divisor), and they satisfy a universal property among maps with this property. The exceptional divisor replaces the center with the projectivization of the normal cone of the center inside the ambient space. Although the construction itself is seemingly simple, blowups have been used extensively throughout the literature: they are the building blocks for birational transformations of smooth projective varieties \cite{AKMW_weakfactorization}, they provide the main desingularization tool \cite{hironaka-resolution}, they allow us to flatify coherent sheaves \cite{GruRay}, and the list goes on.

More recently, a stacky generalization of blowups, called weighted blowups, has gained traction thanks to the work of \cite{abramovich-temkin-wlodarczyk-weighted-resolution} and formalized by \cite{quek-rydh-weighted-blowup}. Intuitively, we can think of them as follows: given a \emph{weighted} center, which is informally a closed subset together with a notion of weights for the normal directions, the \emph{weighted} blowup allows us to replace the weighted center with a Cartier divisor (as in the classical case), but the exceptional divisor is now the \emph{weighted} normal cone, and in general is only a Deligne-Mumford stack (or tame stack in positive characteristic). Weighted blowups have proved to be quite useful. For instance, in \cite{abramovich-temkin-wlodarczyk-weighted-resolution}, the authors proved that Hironaka's desingularization algorithm can be streamlined and improved considerably if one uses this more general theory. They also enable constructions unachievable using classical blowups alone, such as partial resolutions of singularities that preserve the normal crossing locus~\cite{AT-nc,BDS-B-nc,Wlodarczyk-nc}.
Furthermore, they appear constantly in the context of moduli spaces, where the weighted blowup structure has facilitated various enumerative computations \cite{Inchiostro,canning-oprea-pandharipande,Arena-Obinna,battistella-dilorenzo,arena-song}.

Since classical blowups are so prominent in the literature, it was only natural that algebraic geometers started asking themselves how algebro-geometric invariants change under these transformations. When both the ambient space and the center are smooth, explicit \emph{blowup formulas} have been given for derived categories (see \cite{Orlov}), Chow rings (see \cite[Chapter 6.7]{fulton}) and K-theory (see \cite[Exposé VII]{SGA6}, and \cite{Thomason-blowup}). More recently, these results have been generalized to all oriented cohomology theories (see \cite{orientedcohomology_blowup}).

The analogous computations for weighted blowups are more delicate and present further technical complications, due to both the presence of stackiness and the fact that the exceptional divisor is not a projective bundle over the center. The description of the Chow ring of a weighted blowup was obtained in \cite{Arena-Obinna}, while the derived category was computed in \cite{Li}.

Let $\widetilde{\cY} \to \cY$ be a regular weighted blowup of $\cY$ at $\cX$, where $\cY$ and $\cX$ are smooth algebraic stack with affine stabilizers satisfying the resolution property (see \S\ref{sec: notation-resolution} for the definition). In this work, we fully describe the K-theory of weighted blowups in terms of only the K-theory of the ambient variety, the K-theory of the center, and the \emph{weighted Euler class} $\lambda_{-1}^{\Gm}(\sN^\vee_{\cX/\cY})$ of the conormal bundle (see Definition~\ref{def: euler class twisted weighted vector bundle}). Our main result, proved in \S\ref{sec:proof-main-theorem}, is the following.

\begin{Thm}\label{thm: main theorem}
Assume that $\cY$ and $\cX$ are smooth algebraic stack with affine stabilizers satisfying the resolution property. We have an isomorphism of groups
$$
\Kring(\widetilde{\cY}) \simeq \frac{\Kring(\cX)[t^{\pm 1}] \oplus \Kring(\cY)}{\left\{\begin{array}{c} (\gamma \cdot \alpha, i_*\alpha) \quad \forall \alpha \in \Kring(\cX); \\ (\widetilde \alpha,0) \quad \forall \widetilde \alpha \in (p(t^{-1})) \subset \Kring(\cX)[t^{\pm 1}]
\end{array}\right\}}
$$ 
where $p(t^{-1})=\lambda^{\Gm}_{-1}(\sN^\vee_{\cX/\cY}) \in \Kring_{\Gm}(\cX)\simeq \Kring(\cX)[t^{\pm1}]$ and $\gamma = \frac{p(t^{-1})-p(1)}{t-1}$.
  
Moreover, the multiplication rules in $\Kring(\widetilde \cY)$ are as follows: 
\begin{align*}&(q,0)  \cdot (q',0)= (q \cdot q' \cdot (1-t),0),\\ &(0,\beta)\cdot (0,\beta')=(0,\beta\cdot \beta'),\\
&(q,0)\cdot(0,\beta)=(q\cdot i^*\beta,0),
\end{align*}
for every $q,q' \in \Kring(\cX)[t^{\pm 1}]$ and $\beta,\beta' \in \Kring(\cY)$.
\end{Thm}

We obtain this formula by constructing an exact sequence (Corollary~\ref{cor: main exact sequence}) involving only the K-theories of $\tcY$, $\cY$ and $\cX$, analogous to that obtained in Chow in~\cite[Proposition 6.7]{fulton}. Moreover, when the restriction homomorphism $\Kring(\cY)\rightarrow\Kring(\cX)$ is surjective, the formula simplifies considerably and is presented in Corollary~\ref{cor: keel formula}.

The K-theoretic computations present various technical challenges. One subtlety in treating the general case of stacks is that their K-theory is not as well-behaved as that of schemes, even in the smooth setting. For this reason, we restrict ourselves to working with stacks satisfying the resolution property, equivalently, that are global quotient stacks of the form $[X/\GL_n]$, where $X$ is a quasi-affine scheme (see Theorem~\ref{thm: basics on resolution property}). This condition allows us to work with equivariant $K$-theory of schemes, which already appeared in the literature (see \cite{Thomason-equivariantKtheory}).

A more fundamental challenge is that, even for quotient stacks, we cannot use the approximation techniques introduced in \cite{tot} for Chow rings (see also \cite{EG}) to reduce to the case of schemes (or algebraic spaces). To resolve this issue, we extend various K-theoretic constructions to the equivariant setting.

For instance, in \S\ref{sec: gysin and operational} we generalize the definition and some properties of $G$-equivariant Gysin morphisms and operational K-theory for smooth affine group scheme $G$, both introduced in~\cite{AndersonPayne-OperationalKtheory} in the case when $G$ is a split torus. We also prove an \emph{excess intersection formula} (Theorem~\ref{prop: excess intersection formula}) for Gysin morphisms.

In \S\ref{sec:applications}, we present two applications of our main results. First, using the weighted blowup description of the moduli stack of stable genus 1 curves with two marked points established in \cite{Inchiostro}, we compute the K-theory of $\overline{\cM}_{1,2}$. To the best of our knowledge, this yields one of the few explicit computations of the K-theory of moduli stacks currently available in the literature. Besides our result, such computations are known only for $\overline{\cM}_{0,n}$~\cite{orientedcohomology_blowup}, $\overline{\cM}_{1,1}$ (Example~\ref{example M11}), $\cM_2$ and $\overline{\cM}_2$~\cite{Edidin-Hu}.

The K-theory of $\cM_{1,2}$ is proved in \S\ref{sub:K-theory-M-12} and it is given as follows.

\begin{Thm}\label{thm: K-theory M12bar}
    The K-theory of $\overline{\cM}_{1,2}$ is
    {\footnotesize
\[
\Kring(\overline{\cM}_{1,2})\simeq
\frac{\Z[u^{\pm1},s^{\pm1}]}
{\bigl((1-s^{-4})(1-s^{-6})+(1-u^{-4})(1-u^{-6}),
(1-s)(1-u),
(1-u^{-2})(1-u^{-3})(1-u^{-4})\bigr)}.
\]
} 
Moreover, $s=[\sO(-\Delta_1)]$, where
    $\Delta_1$ parametrizes curves with an elliptic tail, and $u=[ \mathbb{E}(\Delta_1)]$ is Hodge bundle twisted by $\Delta_1$.
\end{Thm}

Secondly, in \S\ref{sub:lambda-poly}, we compute the Lambda polynomial of the tangent complex associated to the weighted blowup morphism. 

Set 
$$Q_z(t) := \lambda_z(\sN_{\cX/\cY}) \in \Kring_{\mathbb{G}_m}(\cX)[\![ z ]\!]= \Kring(\cX \times B\Gm)[\![ z ]\!]$$
and let
$$
q: \Kring(\cX)[t^{\pm 1} ] [\![ z ]\!] \oplus \Kring(\cY)[\![ z ]\!] \longrightarrow \Kring(\widetilde{\cY})[\![ z ]\!]
$$
be the quotient map induced by the presentation of $\Kring(\widetilde{\cY})$ in Theorem \ref{thm: main theorem}.

\begin{Thm}\label{thm: Lambda_poly_tangent}
    The identity
    $$
    \lambda_z(\bT_{\widetilde{\cY}/\cY})=\frac{\lambda_z(\bT_{\widetilde{\cY}})}{\lambda_z(\bT_{\cY})}=q \Bigg(  \bigg( \frac{(1+zt^{-1}) Q_z(t)}{(1+z)Q_z(1)}-1 \bigg) \frac{-1}{1-t},1 \Bigg) 
    $$
    holds in $\Kring(\widetilde{\cY})[\![z]\!]$.
\end{Thm}

\subsection{Outline of the paper}

The paper is organized into four sections. 

In \S\ref{sec:weighted-blowups-and-K-theory}, we set the notation and state preliminary results regarding weighted blowups and $G$-equivariant $K$-theory. More specifically, in \S\ref{sub:weighted-blowups} we recall the theory of regular weighted blowups as developed in \cite{quek-rydh-weighted-blowup}, starting
with the ``local'' case, i.e.\! the weighted blowup of $0$ in $\bA^d$. In \S\ref{sub:basics-K-theory} we recall basic results in $G$-equivariant $K$-theory, mainly following \cite{Thomason-equivariantKtheory}, and prove some preliminary results that are essential for the rest of the paper. For instance, we generalize the homotopic invariance of K-theory to twisted weighted vector bundle (Proposition~\ref{prop: homotopy invariance twisted}), which subsequently allows us to prove a version of the splitting principle for the K-theory of twisted weighted vector bundles (see Propositions~\ref{prop:classical-splitting-principle} and \ref{lemma: BT-toBGwn}). In \S\ref{sub:blowup-0-An} we compute explicitly the K-theory for the weighted blowup of $0$ in $\mathbb{A}^d$, which will serve as local case for general blowups.

Section~\S\ref{sec: gysin and operational} is dedicated to generalizing the construction of operational K-theory and the refined Gysin morphism to the $G$-equivariant setting.

Subsequently, in \S\ref{sec:K-theory-weighted-blowups}, we describe the K-theory ring of a weighted blowup. Section~\S\ref{sub: Gysin for Weighted Blowups} is dedicated to describing the Gysin morphism in the weighted blowup case, while \S\ref{sub:blowup-main-theorem} focuses on obtaining the exact sequence key to the proof of Theorem~\ref{thm: main theorem}, to which is dedicated \S\ref{sec:proof-main-theorem}.
We end this section by generalizing Keel's formula to weighted blowups as a corollary of our main result. 

Finally, \S\ref{sec:applications} is dedicated to applications: \S\ref{sub:K-theory-M-12} carries out the computation of the K-theory of $\overline{\cM}_{1,2}$, whereas in \S\ref{sub:lambda-poly} we compute the lambda polynomial of the tangent complex of the blowup morphism.

\subsection{Acknowledgments}

We thank Dan Abramovich, Samir Canning, Zhengning Hu, Oliver Li, Aitor Iribar Lopez, Angelo Vistoli for various discussions about the topic.

During the preparation of the manuscript, V.~A. was supported by EPSRC Horizon Europe Guarantee EP/Y037162/1; A.~C.\ was supported by SNF grant P500PT-222363; A.~L. was partially funded by NSF grant 2401358.

\subsection{Conventions}\label{conventions}
Throughout the paper, we will work over an algebraically closed field $k$.

\subsubsection{$\Gm$-actions and $\Proj$ constructions}

Firstly, we know that every $\Z$-grading on a ring $R$ induces a $\Gm$-coaction on $R$, which in turn induces a $\Gm$-action on $\spec R$. The first association can be made in two natural ways, either preserving the signs of the grading or reversing them. We choose to preserve them in accordance with the literature (see for instance \cite[\S1.1]{quek-rydh-weighted-blowup}).

Let $R=\oplus_{n\in \Z} R_n$ be a graded ring, we denote by $c:R\rightarrow R[t^{\pm 1}]$ the $\Gm$-coaction defined by $c(\{r_n\})=\{r_nt^n\}$ for $r_n \in R_n$. This gives us an induced $\Gm$-action $a:\Gm \times \spec R \rightarrow \spec R$.

\begin{Rmk}
     Let us take for instance $R:= \Z[x_1,\dots,x_d]$ polynomial ring and suppose we give weight $w_i \in \Z$ on the coordinate $x_i$. This induces a grading of $R$ by setting $R_n:= \{ p(x_1,\dots,x_d) \vert \deg^{\bf{w}}p=n\}$
    where $p$ is homogeneous and $\deg^{\bf{w}}p$ is the weighted degree of $p$. The induced coaction $c:R\rightarrow R[t^{\pm 1}]$ can be described by the association $x_i\mapsto x_it^{w_i}$ for every $i$. Unraveling the $\spec $ construction, we get an induced $\Gm$-action 
    $$ a:\Gm \times \A^d \rightarrow \A^d$$
    which is defined by the association $(\lambda,s_1,\dots,s_d) \mapsto (\lambda^{w_1}s_1,\dots,\lambda^{w_d}s_d)$ where $(s_1,\dots, s_d) \in \Gamma(S,\sO_S)^d$ and $\lambda \in \Gamma(S,\sO_S)^*$ for every scheme $S$. 
\end{Rmk}

Moreover, given $\sE$ a locally free sheaf on the algebraic stack $\cX$, we denote by $\mathbb{V}(\sE)$ (respectively $\bP(\sE)$) the algebraic stack $\spec_{\cX}(\Sym^{\bullet}\sE^{\vee})$ (respectively $\Proj_{\cX}(\Sym^{\bullet}\sE^{\vee})$).

Note that we are following Fulton's notation on total spaces for vector (or projective) bundles.

\begin{Def}
    We define $[\bA^1/\Gm]$ where $\Gm$ acts with weight $1$ as the \emph{geometric} standard representation over $B\Gm$. 
\end{Def}

\begin{Rmk}
We use the terminology \emph{geometric standard representation} over $B\Gm$ to distinguish it from the standard algebraic representation corresponding to the weight-$1$ action on the tautological line bundle (as opposed to its total space) on $B\G_m$.
\end{Rmk}

\subsubsection{Quotient stacks and the resolution property}\label{sec: notation-resolution}

We will denote algebraic stacks with curly letters, such as $\cX$, and reserve roman letters for algebraic spaces, such as $X$. 

In this paper, we study the $K$-theory of weighted blowups of a smooth finite-type algebraic stack $\cY$ satisfying the resolution property and with affine stabilizers. Recall that $\cY$ satisfies the resolution property if every coherent $\sO_{\cY}$-module admits a surjection from a locally free sheaf of finite rank. The following characterization of algebraic stacks with the resolution property is due to Totaro.

\begin{Thm}[{\cite[Theorem 1.1]{totaro_resolution},\cite[Theorem A]{Gross_resolution_property}}]\label{thm: basics on resolution property}
   Let $\cX$ be a quasi-compact quasi-separated algebraic stack. Then the following are equivalent:
   \begin{itemize}
       \item the stack $\cX$ has affine stabilizers and it satisfies the resolution property;
       \item the stack $\cX$ is a global quotient stack of the form $[X/\GL_n]$ where $X$ is a quasi-affine scheme.
   \end{itemize}
   In particular, the resolution property is stable under composition with quasi-affine morphisms.
\end{Thm}

In particular, we can always write $\cY=[Y/\GL_n]$ where $Y$ is a smooth quasi-affine scheme.
We will denote by $\cX= [X/\GL_n] \subseteq\cY= [Y/\GL_n]$ the center of the blow up, and by
$$
\widetilde{\cX}= \big[ \tX/\GL_n \big] \subseteq \big[ \tY/ \GL_n \big]= \widetilde{\cY}
$$
the exceptional divisor. We will always assume that $\cX$ is smooth. Here, $\tY$ is the weighted blowup of $Y$ along $X$, with exceptional divisor $\tX$; the above equalities follow from the fact that the formation of weighted blowups commutes with flat base change~\cite[Corollary 3.2.14]{quek-rydh-weighted-blowup}. Note that, even if $\widetilde{Y}$ is not a scheme, we are still using a roman letter to refer to it.

\begin{Rmk}
    The spaces $X$, $Y$, $\cX$, $\widetilde{\cY}$, and $\widetilde{\cX}$ all satisfy the resolution property.

    Indeed, the assertion for $X$ and $\cX$ follows from Theorem~\ref{thm: basics on resolution property}. On the other hand, by Equations \eqref{eqn: weighted-blowup_pres} and \eqref{eqn: GmXGLn-action-def-normal-cone}, the weighted blowup $\widetilde{Y}$ (resp.\ $\widetilde{\cY}$) is a quotient of a smooth Noetherian separated scheme by $\G_m$ (resp.\ $\G_m \times \GL_n$), so \cite[Theorem 5.7]{Thomason-equivariantKtheory} implies that $\widetilde{\cY}$, and thus $\widetilde{\cX}$ as well, satisfies the resolution property.
\end{Rmk}

The reason for working with spaces satisfying the resolution property is to obtain an isomorphism between the K-theory of coherent sheaves and that of locally free sheaves.

\section{Basics on Weighted Blowups and K-theory}\label{sec:weighted-blowups-and-K-theory}
\subsection{Weighted Blowups}\label{sub:weighted-blowups}

Below it will be provided a brief introduction to weighted blowups and of the properties that will be used in the rest of the paper. This section is mostly an extract of the work of Quek and Rydh in \cite{quek-rydh-weighted-blowup}.

Weighted blowups preserve much of the intuition of classical blowups, as they are a birational transformation that replace a given center with a Cartier divisor and it is an isomorphism outside such locus. Many of the usual properties of classical blowups are preserved in a natural way, provided one uses the appropriate generalizations of weighted closed embedding, weighted normal cones, and weighted projective stack bundle, which are defined below.

Before going to the definition used in this paper, we explicitly compute the weighted blowup of $0$ in $\bA^d$, which is the local model of regular weighted blowups.

\subsubsection{Blowup of $0$ in $\mathbb{A}^d$}\label{sub:blowup-0-An}

Let us consider the special case of the weighted blowup of the origin in $\A^d$ with not necessarily distinct weights $\mathbf{w}=(w_1,\dots,w_d)$.

We have a particularly convenient description for the weighted blowup using coordinates. Indeed, setting
$$D_0\bA^d:= \spec\left(\frac{k[x_1,\dots,x_d][s,x_1',\dots,x_d']}{(x_i-s^{w_i}x_i')_{i=1,\dots, d}}\right) \simeq \bA^{d+1},$$
where we assign degree $-1$ to the coordinate $s$ and degree $w_i$ to the coordinate $x_i'$, we can describe the blowup as 
$$Bl_0^\mathbf{w}\bA^d\simeq \Proj\left(\frac{k[x_1,\dots,x_d][s,x_1',\dots,x_d']}{(x_i-s^{w_i}x_i')_{i=1,\dots,d}}\right)=[D_0\bA^d\smallsetminus V(x_1',\dots,x_d')/\Gm]$$
with a natural morphism $Bl_0^\mathbf{w}\bA^d \rightarrow \mathbb{A}^d$ induced by the morphism 
$$
k[x_1,\dots,x_d] \subset \frac{k[x_1,\dots,x_d][s,x_1',\dots,x_d']}{(x_i-s^{w_i}x_i')_{i=1,\dots,d}}
$$ 
of graded algebras. The affine scheme $D_0\bA^d$ coincides with the deformation to the (weighted) normal cone. See below for the general definition of the weighted normal cone. 

Moreover, the exceptional divisor is the principal divisor defined by s=0 and $$V(s)\simeq [\spec(k[x_1',\dots x_d'])\smallsetminus V(x_1',\dots x_d')/\Gm ]\simeq [\bA^d\smallsetminus 0/\Gm] \simeq \cP( w_1, \dots w_d)$$ is a  weighted projective stack. Note that $V(s)$ is not the fiber over $0$ (unless $w_i=1$ for every $i$), rather its reduced structure.

We also have an inclusion of graded algebras
$$ k[x'_1,\dots,x'_d] \subset  \frac{k[x_1,\dots,x_d][s,x_1',\dots,x_d']}{(x_i-s^{w_i}x_i')_{i=1,\dots,d}} $$
and, taking $\Proj$, we get the morphism
$$p:Bl_0^{\bold{w}} \A^d \longrightarrow \cP(w_1,\dots,w_d)$$ and the exceptional divisor $\cP(w_1,\dots,w_d) \subset Bl_0^{\bold{w}} \A^d$ gives us a section of $p$. Similarly to the classical case, the morphism $p$ corresponds to $\sO_{\cP(w_1,\dots,w_d)}(-1)$ and the exceptional divisor corresponds to the $0$-section of the line bundle.

\subsubsection{Working Definitions}

The theory of weighted blowups has been collected and formalized in \cite{quek-rydh-weighted-blowup}. Since we are only interested in regular weighted blowups, we are going to give a different definition, more oriented towards the $K$-theoretical computations. The equivalence of the two definitions can be found in \cite[Proposition 2.18]{italians}. Moreover, we are going to introduce the main differences between weighted and classical blowups the reader have to keep in mind.
 
\begin{Def}\label{def:weighted-blowup}
    Given a regular closed embedding $\cX\subset\cY$ of algebraic stacks of codimension $d$, a regular weighted blowup of $\cY$ in $\cX$ with weights $\mathbf{w}\in \bN^d$ is the datum of a morphism $f:\widetilde{\cY}\rightarrow\cY$ and a Cartier divisor $\widetilde{\cX}\subset\widetilde{\cY}$ such that, smooth locally on $\cY$, it is the pullback of the weighted blowup of $0$ in $\bA^d$. More precisely, there exists a smooth surjective morphism $\pi:\cY'\rightarrow\cY$ and a flat morphism $p:\cY'\rightarrow\bA^d $ such that 
    \begin{itemize}
        \item $p^{-1}(0) = \pi^{-1} (\cX)$ ,
        \item $p^{-1}(Bl_0^\mathbf{w}\bA^d,V(s)) \simeq \pi^{-1} (\widetilde{\cY},\widetilde{\cX})$.
    \end{itemize}
\end{Def}

\begin{Rmk}\label{rmk: blowus are lci}
    Note that the definition forces $f$ to be proper, of finite presentation and lci; furthermore, $f(\widetilde{\cX})=\cX$ and $f$ is indeed an isomorphism over the complement of $\cX$. Moreover, if $\cY$ and $\cX$ are smooth, then the same is true for $\widetilde{\cY}$ and $\widetilde{\cX}$, since in this case $p$ is smooth around $0$. Finally, the weighted blowup is unique up to isomorphism thanks to \cite[Theorem 3.2.9]{quek-rydh-weighted-blowup}.
\end{Rmk}

\begin{Rmk}\label{rem:def-weighted-closed-embedding}
    In the notation of Definition~\ref{def:weighted-blowup}, we can actually construct a weighted closed embedding as in \cite[Definition 4.1.1]{quek-rydh-weighted-blowup}: indeed, in the proof of \cite[Proposition 2.18]{italians}, the authors use the Cartier divisor $\widetilde{\cX}\subset\widetilde{\cY}$ to produce nested ideals $I_m:=f_*\sO_{\widetilde{\cY}}(-m\widetilde{\cX})\subset\sO_{\cY}$. The sequence $\cX_{\bullet}:=V(I_{\bullet})$ is a weighted closed embedding, and it is in fact regular as~\cite[Definition 5.1.3]{quek-rydh-weighted-blowup}. Moreover, we can identify $\widetilde{\cY}$ with $\Proj_{\cY} (R_{I_{\bullet}})$ where $R_{I_{\bullet}}:=\bigoplus_{m \in \bN} I_m$ is a Rees algebra as in \cite[Definition 3.1.1]{quek-rydh-weighted-blowup}. In the same way, $\widetilde{\cX}$ can be identified with $\Proj_{\cX}(\sC_{I_{\bullet}})$ where $\sC_{I_{\bullet}}:=\bigoplus_{m} (I_m/I_{m+1})$ is the weighted conormal algebra. Note that we have a surjective morphism $\sC_{I_{\bullet}} \rightarrow\sO_{\cX}$ of quasi-coherent (graded) commutative $\sO_{\cX}$-algebras, which gives us the zero section $\sigma_{I_{\bullet}}$ of the structural morphism $\cC_{I_{\bullet}}:=\spec_{\cX}\sC_{I_{\bullet}} \rightarrow\cX$. We will refer to the algebraic stack $\cC_{I_{\bullet}}$ (resp. $C_{I_\bullet}$) as the weighted normal cone of $\cX$ in $\cY$ (resp. $X$ in $Y$).
\end{Rmk}

\begin{Rmk}\label{rem:natural-grading}
    In this remark, we explain why a weighted blowup determines ``weighted'' objects, such as the twisted weighted normal cone or the weighted normal bundle. In the notation of Remark~\ref{rem:def-weighted-closed-embedding}, we have the following factorization of the blowup morphism
    $$ \tcY \simeq \Proj_{\cY}(R_{I_{\bullet}})\simeq \left[ \spec R_{I_{\bullet}} \setminus V(R_+)/\Gm \right] \rightarrow \cY \times B \G_m \rightarrow \cY$$
    where $V(R_+)$ is the closed substack defined by the ideal $R_+:=\oplus_{m>0}I_m$. More generally, whenever we refer to \emph{weighted} objects (e.g. weighted normal cone, (twisted) weighted vector bundles, weighted projective bundles, etc...) we imply that these objects live over the product with $B\Gm$, thus possessing a natural grading. 
\end{Rmk}

\begin{Rmk}\label{rem:twisted-weighted}
    We want to stress that $\cC_{I_{\bullet}}$ in general is not a vector bundle on $\cX$, in contrast with what happens in the classical case. Nevertheless, $\cC_{I_{\bullet}}$ is a twisted weighted vector bundle as defined in \cite[Definition 2.1.3]{quek-rydh-weighted-blowup}: this follows from \cite[Proposition 5.1.4]{quek-rydh-weighted-blowup}. More specifically, it is smooth-locally a vector bundle, but not globally (thus the word \emph{twisted}).
\end{Rmk}

Thanks to Remark~\ref{rem:def-weighted-closed-embedding}, we can define an important actor related to the weighted regular embedding. See \cite[Definition 4.2.1]{quek-rydh-weighted-blowup}.
\begin{Def}\label{def:weighted-conormal-sheaf}
     In the notation of Remark~\ref{rem:def-weighted-closed-embedding}, let $(f:\widetilde{\cY}\rightarrow\cY,\widetilde{\cX})$ be a regular weighted blowup of $\cY$ along $\cX$ and $\sC_{I_{\bullet}}$ the weighted conormal algebra on $\cX$ with section $\sigma_{I_{\bullet}}:\sC_{I_{\bullet}} \rightarrow\sO_{\cX}$. We denote by $\sN_{\cX/\cY}^{\vee}$ the weighted conormal bundle associated to $\sigma_{I_{\bullet}}$, i.e. the conormal sheaf $J_{\sigma}/J_{\sigma}^2$ on $\cX$ associated to the ideal $J_\sigma$ of the closed embedding $\sigma_{I_{\bullet}}$. 
\end{Def}

Note that if all the weights are $1$, then $\sN_{\cX/\cY}^{\vee}=I/I^2$ where $I=I_1$ is the ideal associated to the closed embedding $\cX\subset\cY$. Although in the weighted case $\sN_{\cX/\cY}^{\vee}$ is still a vector bundle on $\cX$, it is not necessarily isomorphic to $I/I^2$. The following remark explains how to relate the two objects. See \cite[Remark 5.1.5]{quek-rydh-weighted-blowup} for a more detailed discussion.

\begin{Rmk}
Note that $\sN_{\cX/\cY}^{\vee}$ has a natural grading induced by the one of $\sC_{I_{\bullet}}$, thus it is a graded vector bundle on $\cX$. Moreover, the vector bundle $I/I^2$ on $\cX$ has a filtration $F^{\bullet} \subset (I/I^2)$ where $F^{m}:=(I_m+I^2)/I^2$. There is a natural morphism of graded vector bundles
$$ \sN_{\cX/\cY}^{\vee} \longrightarrow \Gr_{F}(I/I^2)$$
which is in fact an isomorphism (because we are in the regular weighted embedding case). In the classical case, we have that $F^m=0$ for $m>1$ because the algebra $\sC_{I_{\bullet}}$ is generated in degree $1$, thus we have the isomorphism $\sN_{\cX/\cY}^{\vee} \simeq I/I^2$. 
\end{Rmk}

Before talking about the deformation to the weighted normal cone, we point the attention of the reader to another difference between weighted and classical blowups.
\begin{Rmk}\label{rem:fiber-prod-vs-exceptional}
In contrast with the classical case, the exceptional divisor is not a fiber product. That is, the diagram 
$$\begin{tikzcd}
\widetilde{\cX} \arrow[r, "j"] \arrow[d, "g"] & \widetilde{\cY} \arrow[d, "f"] \\
\cX \arrow[r, "i"] & \cY
\end{tikzcd}$$ is commutative but not Cartesian. Nevertheless, the map $\widetilde \cX \ra \cX\times_{\cY} \widetilde \cY$ is a (surjective) closed embedding, identifying $\widetilde{\cX}$ with the reduced underlying substack of $\cX\times_{\cY}\widetilde{\cY}$. Thus, the closed embedding will induced an isomorphism in $K$-theory. 
\end{Rmk}

Finally, it is of great importance to describe the blowup as a quotient of the deformation to the weighted normal cone $D_{\cX_\bullet}\cY$, where $\cX_{\bullet}$ is the weighted closed embedding in $\cY$ induced by the sequence of nested ideals $I_{\bullet}$ induced by the weighted blowup (see Remark~\ref{rem:def-weighted-closed-embedding}). In this paper we will not go over the construction $D_{\cX_\bullet}\cY$ done in \cite[\S4.3]{quek-rydh-weighted-blowup} but we will rather describe its properties. An explicit description of the deformation to the weighted normal cone for the weighted blowup of $0$ in $\bA^d$ can be found in \S\ref{sub:blowup-0-An}. 

Much like the classical deformation to the normal cone, we have a flat map $D_{\cX_\bullet}\cY \ra \bA^1$ (illustrated in Figure \ref{fig:deformation}) such that: 
\begin{itemize}
    \item The fiber over $s=0$ is isomorphic to $(D_{\cX_\bullet}\cY)_0 \simeq \cC_{I_{\bullet}}$;
    \item The fiber over $s \neq 0$ is isomorphic to $(D_{\cX_\bullet}\cY)_s\simeq \cY$;
    \item There is a closed embedding $\cX\times \bA^1 \hookrightarrow D_{\cX_\bullet}\cY$ such that $\cX \times \{0\} \hookrightarrow (D_{\cX_\bullet}\cY)_0 $ is the zero section $\sigma_{I_{\bullet}}$ and $\cX\times \{s\} \hookrightarrow (D_{\cX_\bullet}\cY)_s $ is the immersion $i:\cX\hookrightarrow\cY$ for $s\neq 0$. 
    \item There is a $\Gm$-action on $D_{\cX}\cY$ such that the structural morphism $D_{\cX}\cY \rightarrow \bA^1$ is $\Gm$-equivariant, where $\Gm$ acts on $\bA^1$ with weight $1$. Therefore, we get an induced $\Gm$-action on $(D_{\cX_\bullet}\cY)_0\simeq \cC_{I_{\bullet}}$ and it coincides with the natural grading on $\sC_{I_{\bullet}}$.
\end{itemize}

\begin{figure}[h]
\resizebox{.5\textwidth}{!}{

\tikzset{every picture/.style={line width=0.75pt}} 

\begin{tikzpicture}[x=0.75pt,y=0.75pt,yscale=-1,xscale=1]

\draw   (152.25,218.03) -- (151.87,83.79) -- (218,43) -- (218.37,177.24) -- cycle ;
\draw    (218,43) .. controls (290,86) and (359,65) .. (428,54) ;
\draw    (152.25,218.03) .. controls (234,177) and (338,177) .. (395,198) ;
\draw  [dash pattern={on 0.84pt off 2.51pt}]  (218.37,177.24) .. controls (221.49,173.6) and (264,172) .. (285,173) .. controls (306,174) and (363.12,178.91) .. (378,185) ;
\draw [color={rgb, 255:red, 208; green, 2; blue, 27 }  ,draw opacity=1 ]   (168,158) -- (198,138) ;
\draw [color={rgb, 255:red, 208; green, 2; blue, 27 }  ,draw opacity=1 ]   (400,158) -- (430,138) ;
\draw [color={rgb, 255:red, 208; green, 2; blue, 27 }  ,draw opacity=1 ]   (168,158) -- (400,158) ;
\draw [color={rgb, 255:red, 208; green, 2; blue, 27 }  ,draw opacity=1 ]   (198,138) -- (430,138) ;
\draw   (419.06,196.98) .. controls (390.88,208.01) and (368.02,185.41) .. (367.99,146.49) .. controls (367.96,107.57) and (390.78,67.07) .. (418.96,56.04) .. controls (447.14,45) and (470.01,67.6) .. (470.04,106.52) .. controls (470.06,145.44) and (447.24,185.94) .. (419.06,196.98) -- cycle ;
\draw [color={rgb, 255:red, 74; green, 144; blue, 226 }  ,draw opacity=1 ]   (322,152) -- (298,152) -- (283,152) ;
\draw [shift={(281,152)}, rotate = 360] [color={rgb, 255:red, 74; green, 144; blue, 226 }  ,draw opacity=1 ][line width=0.75]    (4.37,-1.32) .. controls (2.78,-0.56) and (1.32,-0.12) .. (0,0) .. controls (1.32,0.12) and (2.78,0.56) .. (4.37,1.32)   ;
\draw [color={rgb, 255:red, 74; green, 144; blue, 226 }  ,draw opacity=1 ]   (322,145) -- (282,145) ;
\draw [shift={(280,145)}, rotate = 360] [color={rgb, 255:red, 74; green, 144; blue, 226 }  ,draw opacity=1 ][line width=0.75]    (4.37,-1.32) .. controls (2.78,-0.56) and (1.32,-0.12) .. (0,0) .. controls (1.32,0.12) and (2.78,0.56) .. (4.37,1.32)   ;
\draw [color={rgb, 255:red, 74; green, 144; blue, 226 }  ,draw opacity=1 ]   (194,134) -- (198.77,91.99) ;
\draw [shift={(199,90)}, rotate = 96.48] [color={rgb, 255:red, 74; green, 144; blue, 226 }  ,draw opacity=1 ][line width=0.75]    (4.37,-1.32) .. controls (2.78,-0.56) and (1.32,-0.12) .. (0,0) .. controls (1.32,0.12) and (2.78,0.56) .. (4.37,1.32)   ;
\draw [color={rgb, 255:red, 74; green, 144; blue, 226 }  ,draw opacity=1 ]   (185,141) -- (181.13,80) ;
\draw [shift={(181,78)}, rotate = 86.37] [color={rgb, 255:red, 74; green, 144; blue, 226 }  ,draw opacity=1 ][line width=0.75]    (4.37,-1.32) .. controls (2.78,-0.56) and (1.32,-0.12) .. (0,0) .. controls (1.32,0.12) and (2.78,0.56) .. (4.37,1.32)   ;
\draw [color={rgb, 255:red, 74; green, 144; blue, 226 }  ,draw opacity=1 ]   (177,146) -- (169.53,118.93) ;
\draw [shift={(169,117)}, rotate = 74.58] [color={rgb, 255:red, 74; green, 144; blue, 226 }  ,draw opacity=1 ][line width=0.75]    (4.37,-1.32) .. controls (2.78,-0.56) and (1.32,-0.12) .. (0,0) .. controls (1.32,0.12) and (2.78,0.56) .. (4.37,1.32)   ;
\draw [color={rgb, 255:red, 74; green, 144; blue, 226 }  ,draw opacity=1 ]   (276.52,103.54) .. controls (290.95,117.34) and (317.6,121.9) .. (334,119) ;
\draw [shift={(275,102)}, rotate = 46.97] [color={rgb, 255:red, 74; green, 144; blue, 226 }  ,draw opacity=1 ][line width=0.75]    (4.37,-1.32) .. controls (2.78,-0.56) and (1.32,-0.12) .. (0,0) .. controls (1.32,0.12) and (2.78,0.56) .. (4.37,1.32)   ;
\draw [color={rgb, 255:red, 74; green, 144; blue, 226 }  ,draw opacity=1 ]   (262.49,112.41) .. controls (279.03,127.17) and (311.63,129.88) .. (332,126) ;
\draw [shift={(261,111)}, rotate = 45] [color={rgb, 255:red, 74; green, 144; blue, 226 }  ,draw opacity=1 ][line width=0.75]    (4.37,-1.32) .. controls (2.78,-0.56) and (1.32,-0.12) .. (0,0) .. controls (1.32,0.12) and (2.78,0.56) .. (4.37,1.32)   ;
\draw    (119,240.5) -- (479,239.5) ;
\draw    (153,245.5) -- (153,235.5) ;

\draw (163,33.4) node [anchor=north west][inner sep=0.75pt]    {$C_{I_{\bullet}}$};
\draw (402,161.4) node [anchor=north west][inner sep=0.75pt]    {$X$};
\draw (170,161.4) node [anchor=north west][inner sep=0.75pt]    {$X$};
\draw (468,53.4) node [anchor=north west][inner sep=0.75pt]    {$\mathrm{Y}$};
\draw (139,249.4) node [anchor=north west][inner sep=0.75pt]  [font=\footnotesize]  {$s=0$};
\draw (229,78.4) node [anchor=north west][inner sep=0.75pt]  [font=\small]  {$\textcolor[rgb]{0.29,0.56,0.89}{\mathbb{G}_{m} \circlearrowright }$};

\end{tikzpicture}
}

\label{fig:deformation}
\end{figure}

We can obtain the weighted blowup as 
\begin{equation}\label{eqn: weighted-blowup_pres}
    \widetilde\cY \simeq \left[\left( D_{\cX_\bullet}\cY \smallsetminus (\cX \times \bA^1)\right)/\Gm\right].
\end{equation}

\begin{Rmk}\label{rem:blowup-of-def-cone}
    The weighted (regular) embedding $\cX_{\bullet} \subset \cY$ induces a weighted (regular) embedding $\cX_{\bullet} \times \A^1 \subset D_{\cX_{\bullet}}\cY$ which in turn gives rise to a weighted blowup $b_D: \widetilde{\cD} \rightarrow D_{\cX_{\bullet}}\cY$ on $\bA^1$. Since the sequence of closed substacks $\cX_{\bullet} \times \A^1$ is flat over $\A^1$, restricting over a point $t \in \A^1$ commutes with the (weighted) blowup construction, therefore we get 
    \begin{itemize}
        \item if $t\neq 0$, then $(b_D)_t$ coincides with the weighted blowup $\widetilde{\cY} \rightarrow \cY$ of $\cX$ in $\cY$;
        \item if $t=0$, then $(b_D)_0$ coincides with the weighted blowup of $\cX$ in the normal cone $\cC_{I_{\bullet}}$ as the (weighted) zero section. 
    \end{itemize}
\end{Rmk}
\subsection{K-theory}\label{sub:basics-K-theory}

In this subsection we review the main definitions and properties of algebraic K-theory that are essential for the rest of the paper. For a more detailed exposition, see the foundational papers~\cite{Thomason-LefschetzRR,Thomason-equivariantKtheory,Thomason-blowup}, or the more recent~\cite{Merkurjev-equivariantKtheory,AndersonPayne-OperationalKtheory}.

Throughout the subsection, $X$ will always denote a separated scheme of finite type over $k$, and $G$ a smooth affine finite type group scheme over $k$ acting on $X$.

\subsubsection{Basic Definitions}\label{sec: basic-def-K-theory}

We denote by $\Kgroup^G(X)$ the Grothendieck group of $G$-equivariant coherent sheaves on $X$. This is generated by classes $[F]$ of $G$-equivariant coherent sheaves subject to the relations $[F]=[F_1]+[F_2]$ for every exact sequence
\[
\begin{tikzcd}
    0\arrow[r] & F_1\arrow[r] & F\arrow[r] & F_2\arrow[r] & 0
\end{tikzcd}
\]
of equivariant coherent sheaves. The group operation is induced by the direct sum.

Similarly, we denote by $K_G^{\circ}(X)$ the Grothendieck group generated by $G$-equivariant perfect complexes, subject to the same relations. The derived tensor product endows $\Kring_G(X)$ of a ring structure and makes $\Kgroup^G(X)$ into a $\Kring_G(X)$-module.

The formation of $\Kgroup^G$ is covariant for $G$-equivariant proper morphisms $f:X\rightarrow Y$, with $f_*:\Kgroup^G(X)\rightarrow\Kgroup^G(Y)$ induced by
\[
    f_*[F]=\sum_{i}(-1)^i[\mathrm{R}^if_*F].
\]

Moreover, every $G$-equivariant morphism $f:X \to Y$ induces $f^*:\Kring_G(Y)\rightarrow\Kring_G(X)$ via the derived pullback.
For proper morphisms $f$, there is a projection formula (\cite[\href{https://stacks.math.columbia.edu/tag/01E8}{Tag 01E8}]{stacks-project})
\begin{equation}\label{eqn: Projection Formula}
f_*( f^*(\alpha) \cdot \beta)= \alpha \cdot f_*(\beta) \ \text{for} \ \alpha \in \Kring_G(Y) \ \text{and} \  \beta \in \Kgroup^G(X).
\end{equation}
In particular, the pushforward $f_*:\Kgroup^G(X) \to \Kgroup^G(Y)$ is a homomorphism of $\Kring_G(Y)$-modules.

When $f$ is flat, the pullback $f^*: \Kring_G(Y)\rightarrow\Kring_G(X)$ coincides with the usual pullback of (a representative of) the complex. In fact, if $f$ is flat there is also a flat pullback $f^*:\Kgroup^G(Y) \to \Kgroup^G(Y)$ defined in the same way.

When $X=\spec k$, we have natural isomorphisms
\[
    \Kgroup^G(\spec k)\simeq\Kring_G(\spec k)\simeq R(G),
\]
where $R(G)$ is the representation ring of $G$. In particular, for any $X$, the flat pullback along $X\rightarrow\spec k$ makes $\Kring_G(X)$ into an $R(G)$-algebra and $\Kgroup^G(X)$ into an $R(G)$-module.

\begin{Rmk}\label{rmk: vector bundles vs locally-free sheaves}
    Sometimes it is convenient to identify a locally-free sheaf $\sE$ on $X$ with its associated geometric vector bundle $E=\spec\Sym^{\bullet}(\sE^\vee)$. Here, $\sE$ is recovered from $E$ as the sheaf of sections of the structure map $E\rightarrow X$. Then, we set $[E]:=[\sE]\in\Kring_G(X)$.
\end{Rmk}

\subsubsection{Comparison Isomorphisms} \label{sec: comparison isomorphisms}

In general, the two K-theory groups defined in \S\ref{sec: basic-def-K-theory} are not isomorphic, and both differ from the Grothendieck group $K_G^{\mathrm{vect}}(X)$ of $G$-equivariant vector bundles. However, they coincide under suitable assumptions, that we recall now.

When $X$ is equivariantly embeddable in a regular Noetherian separated scheme, then $\Kring_G(X)\simeq K_G^{\mathrm{vect}}(X)$, as every equivariant coherent sheaf on $X$ can then be equivariantly resolved by a (possibly infinite) complex of locally-free sheaves, by~\cite[\S2]{Thomason-EquivariantResolution}. See also \S\ref{sec: notation-resolution}.

If $X$ is itself smooth, then the forgetful map $\Kring_G(X)\rightarrow\Kgroup^G(X)$ is an isomorphism of $\Kring_G$-modules, as in this case every equivariant coherent sheaf on $X$ admits a \emph{finite} resolution by locally-free sheaves, see~\cite[Theorem 5.7]{Thomason-equivariantKtheory}. In this case, $\Kring_G$ becomes a covariant functor via pushforward with respect to equivariant proper morphisms.

In general, the group $\Kgroup^G(X)$ can be interpreted as the Grothendieck group of coherent sheaves on the quotient stack $[X/G]$. In particular, when the action of $G$ on $X$ is free with quotient an algebraic space $X/G$, there is a natural isomorphism $\Kgroup^G(X)\simeq\Kgroup(X/G)$. This also shows that if $H$ is another algebraic group over $k$ acting on $Y$, and $[X/G]\simeq[Y/H]$, then $\Kgroup^G(X)\simeq\Kgroup^H(Y)$. The same discussion applies to $\Kring$.

\subsubsection{Representation Rings and Splitting Principle} \label{sec: representation rings}

We present and compare two simple examples of representation rings that will play a role in the paper.

The simplest is the representation ring of split tori. Denote by $t$ the class of the geometric standard representation of $\Gm$, or equivalently the sheaf of sections of the universal line bundle $[\A^1/\Gm]\rightarrow B\Gm$, where the action is with weight 1. Then,
\begin{equation}\label{eq: representation ring tori}
    \Kgroup^{\Gm^n}(\spec k)\simeq\Kring_{\Gm^n}(\spec k)\simeq R(\Gm^n)\simeq\Z[t_1^{\pm1},\ldots,t_n^{\pm n}],
\end{equation}
where $t_i$ is the pullback of $t$ along the $i$-th projection $B\Gm^n\rightarrow B\Gm$ (\cite[Theorem 12.12, \S12.e]{Milne_algebraicgroups}).

The second example is the representation ring of $\GL_n$, which can be described as follows. Let $V$ be the geometric standard representation of $\GL_n$, and set $e_i=[\Lambda^iV]$. Then, over any field,
\begin{equation}\label{eq: representation ring GLn}
    \Kring(B\GL_n)=R(\GL_n)\simeq\Z[e_1,\ldots,e_{n-1},e_n^{\pm 1}].
\end{equation}
Over the complex numbers, this is~\cite[Exercise 23.36 (d)]{FultonHarris-RepresentationTheory}. In general, the above isomorphism follows from equation~\eqref{eq: representation ring tori} and~\cite[Theorem 22.38]{Milne_algebraicgroups}. The morphism $R(\GL_n)\rightarrow R(\Gm^n)$ to the representation ring of the maximal torus of $\GL_n$ consisting of diagonal matrices sends $e_i$ to the symmetric function of degree $i$ in the classes $t_1,\ldots,t_n$.

\begin{Lemma}\label{lemma: basis R(GL)}
    $R(\Gm^n)$ is a free $R(\GL_n)$-module of finite rank.
\end{Lemma}
\begin{proof}
    This follows from the identity $R(\Gm^n)\simeq R(\GL_n) \otimes_{\Z[e_1,\ldots,e_{n-1},e_n^{\pm1}]}\Z[t_1,\dots,t_n]$ and \cite[Lemma 3.1]{CLI}.
\end{proof}

More generally, if $G$ is a split reductive group with simply connected commutator group and maximal torus $S$, the restriction of the action to $S$ induces a morphism
$
\Kgroup^{G}(X) \to \Kgroup^{S}(X),
$
and the map
\begin{equation}\label{eqn: iso from Gln to T}
\Kgroup^{G}(X) \otimes_{R(G)} R(S) \longrightarrow \Kgroup^{S}(X)
\end{equation}
is an isomorphism of $R(S)$-modules by \cite[Proposition~31]{Merkurjev-equivariantKtheory}. In particular, the restriction map $
\Kgroup^{G}(X) \to \Kgroup^{S}(X)$ is injective.

This yields a version of the splitting principle in K-theory.

\begin{Prop}[The splitting principle] \label{prop:classical-splitting-principle}
Assume that $G$ is split reductive and with simply connected commutator group, and let $T\simeq\Gm^n$ be the standard maximal torus in $\GL_n$. Fix a $G$-invariant map $X \to B \GL_n$. Then the $G$-equivariant map 
\[
\begin{tikzcd}
    X'=X \times_{B\GL_n} BT \arrow[r] & X
\end{tikzcd}
\]
in the following fiber diagram $$\begin{tikzcd}
X' \arrow[r] \arrow[d] & BT \arrow[d]\\
X \arrow[r] & B\GL_n  \\
\end{tikzcd}$$
induces an injection of K-theories $\Kgroup^G(X) \hookrightarrow \Kgroup^G(X')$ via flat pullback.
\end{Prop}
\begin{proof}
Let $S$ be the maximal torus of $G$. Let $P$ be the principal $G$-equivariant $\GL_n$-bundle associated to $X\rightarrow B\GL_n$. Then $X\simeq P/\GL_n$ and $X'\simeq P/T$. In particular,
\[
\Kgroup^G(X)\simeq \Kgroup^{\GL_n\times G}(P),\quad\text{and}\quad \Kgroup^G(X')\simeq \Kgroup^{T\times G}(P).
\]
The composition
\[
\Kgroup^{\GL_n \times G}(P) \longrightarrow \Kgroup^{T \times G}(P) \longrightarrow \Kgroup^{T \times S}(P)
\]
is injective by \cite[Proposition 31]{Merkurjev-equivariantKtheory}. In particular, the first map, namely the pullback
\[
\Kgroup^{G}(X) \longrightarrow \Kgroup^{G}(X')
\]
is injective.
\end{proof}

\subsubsection{Excision and Non-reduced Structures} \label{sec:excision}

We are now ready to record some key properties of K-theory that closely parallel those of Chow groups. We begin with the excision property, more commonly referred to as the localization sequence.

Given a $G$-equivariant closed immersion of schemes $Z\hookrightarrow X$, there is an exact sequence
\[
\begin{tikzcd}
    \Kgroup^G(Z)\arrow[r] & \Kgroup^G(X)\arrow[r] & \Kgroup^G(X\smallsetminus Z)\arrow[r] & 0.
\end{tikzcd}
\]
This can be extended on the left by means of higher K-theory. An immediate application of the extended sequence is that $\Kgroup(X)\simeq\Kgroup(X_{\mathrm{red}})$, meaning that K-theory does not detect non-reduced structures. This is the only point where we will use the extended exact sequence.

\subsubsection{Homotopy Invariance}\label{sec: homotopy invariance}

Another useful property is that if $\pi:V\rightarrow X$ is a $G$-equivariant vector bundle over $X$, then $\pi^*$ induces an isomorphism between K-theories. More generally, given a $V$-torsor $f:E\rightarrow X$ with a $G$-action making both the map $f$ and the $V$-action $G$-equivariant, the pullback $f^*:\Kgroup^G(X)\rightarrow\Kgroup^G(E)$ is an isomorphism, see~\cite[Theorem 4.1]{Thomason-equivariantKtheory}.
This can be generalized to twisted weighted vector bundles as follows. First, we need the following lemma, which gives a dual picture to~\cite[Proposition 3.8]{Arena-Obinna}.

\begin{Lemma}\label{lem: factorization of twisted weighted vector bundles}
    Let $\pi:E\rightarrow X$ be a twisted weighted vector bundle with positive weights $w_1<\ldots<w_r$, and let $n_i$ be the dimension of the eigenspace corresponding to the eigenvalue $w_i$ of the $\Gm$-action. Suppose that $G$ acts on both $E$ and $X$ so that $\pi$ is $G$-equivariant, and that the $G$-action on $E$ commutes with the natural $\Gm$-action. Then, there exists a factorization
    \begin{equation}\label{eq: factorization of twisted weighted vector bundles}
        \begin{tikzcd}
            E=Q_r\arrow[r,"\phi_r"] & Q_{r-1}\arrow[r,"\phi_{r-1}"] & \ldots \arrow[r] & Q_1\arrow[r,"\phi_1"] & Q_0=X,
        \end{tikzcd}
    \end{equation}
    and for $1\leq i\leq r$ there exists a $G$-equivariant vector bundle $V_i\rightarrow X$ of rank $n_i$ over $X$ and a $G$-action on $Q_i$, such that:
    \begin{itemize}
        \item the $G$ action on $E$ and $X$ are the ones we started with;
        \item for $1\leq i\leq r$ , the map $\phi_{i}$ is a $G$-equivariant torsor under the vector bundle $V_i\times_{X}Q_{i-1}$ over $Q_{i-1}$, and the $V$-action is $G$-equivariant.
    \end{itemize}
\end{Lemma}
\begin{proof}
    Let $U\subset X$ be an open subscheme such that
    \[
        E|_U\simeq\spec\sO_{U}[x_1^{(w_1)},\ldots,x_{n_1}^{(w_1)},\ldots,x_1^{(w_r)},\ldots, x_{n_r}^{(w_r)}],
    \]
    where $x_i^{(w_l)}$ has weight $w_l$, for every $i$ and $l$. For every $l<r$, let
    \[
        R^U_l:=\sO_U[x_1^{(w_1)},\ldots,x_{n_1}^{(w_1)},\ldots,x_1^{(w_l)},\ldots,x_{n_l}^{(w_l)}]
    \]
    be the $\sO_U$-sub-algebra generated by elements of degree less than or equal to $w_l$. By construction, for any open subschemes $U_1,U_2\subset X$, the transition functions of $E$ restrict to transition functions $R_l^{U_1}|_{U_2}\xrightarrow{\simeq} R_l^{U_2}|_{U_1}$ that satisfy the cocycle relation. This yields an $\sO_X$-sub-algebra $R_l\subset R$ with $R_l|_U\simeq R_l^U$. Then $Q_l:=\spec_{X}\! R_l\rightarrow X$ is again a twisted weighted vector bundle, with weights $w_1<\ldots<w_l$. Since the $G$-action on $E$ commutes with the $\Gm$-action, it respects the weights and thus restricts to a $G$-action on $Q_l$. This yields the sequence~\eqref{eq: factorization of twisted weighted vector bundles} that $G$-equivariantly factors $E$ into twisted weighted vector bundles over $X$. We are left to prove that there exists a vector bundle $V_r$ over $X$ such that $\phi_r:E\rightarrow Q_{r-1}$ is a torsor under $V_r\times_{X}Q_{r-1}$, whose action map is $G$-equivariant. The same statement for $\phi_i$ can be either proved analogously or by using induction.
    
    For every pair of open subschemes $U_1,U_2\subset X$ that trivialize $E$, let $\underline{x}$ denote the variables with weight less than $w_r$, and $\underline{y}$ the variables with weight $w_r$. Note that $R|_{U_1}\simeq R_{r-1}|_{U_1}[\underline{y}]$, and similarly for $R|_{U_2}$, with the transition map being $\underline{y}\mapsto A\underline{y}+P(\underline{x})$ for a matrix $A$ and polynomial $P$ both with coefficients in $\sO_{U_1\cap U_2}$. Note that $P(\underline{x})$ is a local section of $R_{r-1}$ that defines the affine translation. The cocycle condition on $E$ implies that the matrices $A$ satisfy the cocycle condition with $U_1$ and $U_2$ varying, thus giving a vector bundle $V_r\rightarrow X$ over $X$. This local description shows that $\phi_r:E\rightarrow Q_{r-1}$ is an affine bundle modeled on the pullback $V_r\times_X Q_{r-1}$, making it a $V_r\times_X Q_{r-1}$-torsor. The $V_r\times_X Q_{r-1}$-action on $Q_r$ is compatible with the $G$-action by construction and the fact that the $G$-action on $E$ commutes with the $\Gm$-action.
\end{proof}

\begin{Rmk}\label{rmk: mistake in Arena-Obinna}
    Note that the factorization~\eqref{eq: factorization of twisted weighted vector bundles} we constructed in Lemma~\ref{lem: factorization of twisted weighted vector bundles} is induced by the filtration of the $\sO_X$-algebra $R$ as defined in the proof of~\cite[Proposition 3.8]{Arena-Obinna}. However, in loc.\! cit.\! the authors claim that this provides a filtration of $E$ in twisted weighted sub-bundles rather than a factorization as above, which is not true. Nevertheless, the statement of~\cite[Proposition 3.8]{Arena-Obinna} still holds. That is, there exist weighted vector bundles $E_l$ on $X$ for $l=1,\ldots,r$ and fibrations
    $$
    F_{l-1} \hookrightarrow F_l \to E_{l}
    $$
    where each $F_l$ is a twisted weighted vector bundle on $X$ and such that $F_{r}=E$ and $F_0=X$.    
    
    Since the existence of such a filtration is relevant to this discussion, we sketch a corrected proof of loc.\! cit.\!
    We use the same notation as in the proof of Lemma~\ref{lem: factorization of twisted weighted vector bundles}. Let $U\subset X$ be an open subscheme trivializing $E$, and for every $1\leq l<r$ define
    \[
        \overline{R}^U_l:=\sO_{U}[\overline{x}_1^{(w_{r+1-l})},\ldots,\overline{x}_{n_{r+1-l}}^{(w_{r+1-l})},\ldots,\overline{x}_1^{(w_r)},\ldots,\overline{x}_{n_r}^{(w_r)}]
    \]
    be the $\sO_U$-algebra obtained by taking the quotient of $R|_U$ by the variables $x_i^{(w_i)}$ with weight less than $r+1-l$. By construction, for any $U_1,U_2\subset X$ trivializing $E$, the transition functions on $E$ restrict to isomorphisms $\overline{R}_l^{U_1}|_{U_1 \cap U_2}\xrightarrow{\simeq}\overline{R}_l^{U_2}|_{U_1 \cap U_2}$ satisfying the cocycle relation. This yields a $\sO_X$-algebra $\overline{R}_l$ with a projection $R\rightarrow\overline{R}_l$ such that its restriction to any open $U$ trivializing $E$ coincides with $R|_U\rightarrow\overline{R}_l^U$. Moreover, $F_l:=\spec_{X}\!\overline{R}_l\hookrightarrow E$ is a twisted weighted vector sub-bundle, with $F_0:=X$ and $F_r=E$. We claim that for every $1\leq l\leq r$ there exists a vector bundle $E_l\rightarrow X$ of rank $n_{r+1-l}$ and a fibration $F_{l-1}\hookrightarrow F_l\rightarrow E_l$ over $X$. By induction, it is enough to show this for $l=r$. Let $U\subset X$ trivializing $E$, and set $E_r^U=\spec\sO_U[x_1^{(w_1)},\ldots,x_{n_1}^{(w_1)}]$. By construction, the transition maps for $E$ preserve the $E_r^U$, are linear and satisfy the cocycle relation, thus giving a vector bundle $E_r$ of rank $n_1$ over $X$. Then, $F_{r-1}\hookrightarrow F_r\rightarrow E_r$ is the desired fibration. Note that this is canonically defined; in particular, if $E$ and $X$ admits a $G$-action that commutes with the $\Gm$-action on $E$ and makes $E\rightarrow X$ equivariant, also $F_i$ and $E_i$ admit natural $G$-actions making all morphisms between them $G$-equivariant.

    We remark that the sequence of vector bundles over $X$ obtained from this filtration are the same as those obtained from the factorization Lemma~\ref{lem: factorization of twisted weighted vector bundles}. More precisely, $E_l\simeq V_{r+1-l}$. This follows from the fact that their transition functions are defined by the same matrix.
    Another canonical way of obtaining the weighted vector bundles $E_i$ is to consider the sheaf $\sE:=R_+/R_+^2$, which is a weighted vector bundle, see~\cite[\S2.1.9]{quek-rydh-weighted-blowup}. Then, the $\Gm$-action on $\sE$ induces a splitting $\sE=\oplus\sE_i$, and $E_i=\spec\Sym^{\bullet}\sE_i$. Equivalently, the sheaf of section of $E_i$ is isomorphic to $\sE_i^{\vee}$ and thus $E_i$ is the total space of $\sE_i^{\vee}$.
\end{Rmk}

\begin{Prop}\label{prop: homotopy invariance twisted}
    Let $\pi:E\rightarrow X$ be a twisted weighted vector bundle with positive weights $w_1<\ldots<w_r$. Suppose that $G$ acts on both $E$ and $X$ so that $\pi$ is $G$-equivariant, and that the $G$-action on $E$ commutes with the natural $\Gm$-action. Then, $\pi^*:\Kgroup^G(X)\rightarrow\Kgroup^G(E)$ is an isomorphism.
\end{Prop}
\begin{proof}
    By Lemma~\ref{lem: factorization of twisted weighted vector bundles}, we can factor $E\rightarrow X$ as the composition of torsors under vector bundles. Then, the statement follows by iteratively applying~\cite[Theorem 4.1]{Thomason-equivariantKtheory}.
\end{proof}

It is important to note that in general the above results do not hold for $\Kring$ if $X$ is not smooth even in the case of vector bundles, see for instance~\cite[Example 2.1]{AndersonPayne-OperationalKtheory}.

\subsubsection{Splitting principle for twisted weighted vector bundles}

As a corollary of Proposition~\ref{prop: homotopy invariance twisted}, we obtain a splitting principle for twisted weighted vector bundles.

Let $\mathbf{w}=(w_1,\ldots,w_r)$ be a sequence of positive integers satisfying
$w_1<\cdots<w_r$, and let $\mathbf{n}=(n_1,\ldots,n_r)$ be a sequence of positive integers. As in the case of vector bundles, a $G$-equivariant twisted weighted vector bundle $E$ on $X$ with weights $\mathbf{w}$ and such that the eigenspace of weight $w_i$ has rank $n_i$ for each $i=1,\ldots,r$ is equivalent to giving a $G$-invariant morphism
$$
X\longrightarrow BG_{\mathbf{w},\mathbf{n}}.
$$
The group $G_{\mathbf{w},\mathbf{n}}$ can be described inductively on $r$ as follows.

\begin{Def}[{\cite[\S2.1.7]{quek-rydh-weighted-blowup}}]
    If $r=1$, we have that $G_{w_1,n_1}:= \GL_{n_1}$. If $r>1$, let us denote by $\mathbf{w}'$ (respectively $\mathbf{n}'$) the vector of integers obtained by forgetting the last entry of $\mathbf{w}$ (respectively $\mathbf{n}$). We have the isomorphism of groups
    \[
        G_{\mathbf{w},\mathbf{n}}:= (\GL_{n_r} \times G_{\mathbf{w'},\mathbf{n}'}) \ltimes \G_a^{n_rN_r}
    \]
    where $N_r$ is the dimension of the $w_r$-th degree piece of a graded polynomial algebra with free variables $\{x_{i,j} : 1 \leq i\leq r -1, 1\leq j\leq n_i\}$, where $x_{i,j}$ is given weight $w_i$.
\end{Def}

We briefly explain why morphisms $X\rightarrow BG_{\mathbf{w},\mathbf{n}}$ parametrize twisted weighted vector bundles $E$ on $X$ with weight vector $\mathbf{w}$, where the eigenspace of each weight $w_i$ has rank $n_i$ for $i=1,\ldots,r$ (see also \cite[\S2.1.7]{quek-rydh-weighted-blowup}). When $r=1$, this follows from the fact that every twisted weighted vector bundle with a single weight has linear transition functions. If $r>1$, it essentially follows from the proof of Lemma~\ref{lem: factorization of twisted weighted vector bundles}. With the notation as in Lemma~\ref{lem: factorization of twisted weighted vector bundles}, for every pair of open subschemes $U_1,U_2\subset X$ that trivialize $E$, let $\underline{x}$ denote the variables with weight less than $w_r$, and $\underline{y}$ the variables with weight $w_r$. Note that $R|_{U_1}\simeq R_{r-1}|_{U_1}[\underline{y}]$, and similarly for $R|_{U_2}$, with the transition map being $\underline{y}\mapsto A\underline{y}+P(\underline{x})$ for a matrix $A$ and polynomial $P$ both with coefficients in $\sO_{U_1\cap U_2}$. The coefficients of the polynomial $P$ can be identified with an element of $ \Ga^{n_rN_r}(U_1\cap U_2)$. 

\begin{Rmk}\label{rmk: natural_maps_G_Gln}
    From the definition of $G_{\mathbf{w},\mathbf{n}}$ we obtain a natural quotient map $G_{\mathbf{w},\mathbf{n}} \to  \GL_{\mathbf{n}}$. The corresponding morphism
    $$
    B G_{\mathbf{w},\mathbf{n}} \longrightarrow B\GL_{\mathbf{n}}
    $$
    is precisely given by the association $E\mapsto \{V_i\}_{i=1,\dots,r}$ as in Lemma~\ref{lem: factorization of twisted weighted vector bundles}. Geometrically, this morphism corresponds to isolating the transition functions of the component vector bundles $V_i$.

    Moreover, by construction there is also a natural inclusion of groups $\GL_{\mathbf{n}}\rightarrow G_{\mathbf{w},\mathbf{n}}$ which gives us a section of the gerbe
    $$
    BG_{\mathbf{w},\mathbf{n}} \rightarrow B \GL_{\mathbf{n}}.
    $$
\end{Rmk}

\begin{Rmk}\label{rem: V-torsor for semidirect product}
    Let $G$ be a semi-direct product $Q\ltimes\Ga^n$ where $Q$ is a linear group scheme.
    Denote by $V_G$ the vector bundle $[\Ga^n/G]$ on $BG$ obtained from the conjugation action of $G$ on $\Ga^n$. Because the conjugation action is linear, this is indeed a vector bundle over $BG$. Define a morphism
    \[
        \rho: V_G \times_{BG}BQ \longrightarrow BQ
    \]
    as follows. Let $P_G\rightarrow S$ be a $G$-torsor over $S$ with a $G$-equivariant map $P_G\rightarrow\Ga^n$, $P_Q\rightarrow S$ a $Q$-torsor, and $\lambda$ an isomorphism over $S$ between $P_G$ and $(P_Q \times G )/Q$. The map $\rho$ associates to this triple the $Q$ torsor $P_G/\Ga^n$ over $S$. Note that this comes with a natural isomorphism $P_G/\Ga^n \xrightarrow{\sim} (P_Q \times (G/\Ga^n))/Q$ and thus with a natural isomorphism with $P_Q$. In other words, there is a natural transformation between $\rho$ and the projection  $ V_G \times_{BG}BQ \longrightarrow BQ$. We will ignore this subtlety and identify $\rho$ with this projection in the sequel. In particular, $\rho$ defines a $V_G$-action on $BQ$ relative to $BG$.
\end{Rmk}
\begin{Lemma}\label{lem: V-torsor for semidirect product}
    In the notation of Remark~\ref{rem: V-torsor for semidirect product}, the map $BQ\rightarrow BG$ is a $V_G$-torsor.
\end{Lemma}
\begin{proof}
     Since $\rho$ is a $V_G$-action, it is enough to prove that the induced morphism 
    $$ (\rho,\rho):V_G \times_{BG} BQ \longrightarrow BQ\times_{BG}BQ$$
    is an isomorphism. Identifying  
    $$
    BQ \times_{BG} BQ \simeq (BQ \times BQ) \times_{BG \times BG} BG
    $$
    and then the diagonal $BG \to BG \times BG$ with the morphism $[G/G \times G] \to BG \times BG$ given by the two projections, one sees that 
    $$
    BQ \times_{BG} BQ \simeq  [G/Q \times Q].
    $$
    Here $G\times G$ acts on $G$ by
    \[
        (g_1,g_2)\cdot x = g_1xg_2^{-1},
    \]
    and this action restricts to an action of $Q\times Q$ on $G$.

    Furthermore, one can identify the map $  V_G \times_{BG}BQ \simeq [\Ga^n/Q]$ and the map $(\rho,\rho)$ with the morphism
    $$
    [\Ga^n/Q] \longrightarrow [G/ Q\times Q]
    $$
    induced by the natural inclusion $\Ga^n \hookrightarrow G$. Note that the inclusion is equivariant with respect to the morphism of groups $Q\rightarrow Q\times Q$ given by the diagonal embedding. Checking that $(\rho,\rho)$ is an isomorphism is now straightforward.
\end{proof}

\begin{Prop}\label{lemma: BT-toBGwn}
Let $X \to BG_{\mathbf{w},\mathbf{n}}$ be a $G$-invariant morphism corresponding to a $G$-equivariant twisted bundle $E$ on $X$. Consider the Cartesian diagram
\[
\begin{tikzcd}
X' \arrow[r] \arrow[d] & B \GL_{\mathbf{n}} \arrow[d] \\
X \arrow[r] & B G_{\mathbf{w},\mathbf{n}}.
\end{tikzcd}
\]
induced by the map $ B \GL_{\mathbf{n}} \longrightarrow BG_{\mathbf{w},\mathbf{n}}$ in \ref{rmk: natural_maps_G_Gln}. Then the $G$-equivariant morphism $X' = X \times_{BG_{\mathbf{w},\mathbf{n}}} B\GL_{\mathbf{n}} \longrightarrow X$ induces an isomorphism through flat pullback $\Kgroup^G(X) \rightarrow \Kgroup^G(X')$.
Moreover, the pullback of $E$ to $X'$ is isomorphic to the direct sum of the pullback of the vector bundles $E_i$ introduced in in Lemma~\ref{lem: factorization of twisted weighted vector bundles}.
\end{Prop}

\begin{proof}
    The morphism $ B \GL_{\mathbf{n}} \longrightarrow BG_{\mathbf{w},\mathbf{n}}$ in Remark~\ref{rmk: natural_maps_G_Gln} is a composition of torsors under vector bundles by Lemma~\ref{lem: V-torsor for semidirect product}. Applying \cite[Theorem 4.1]{Thomason-equivariantKtheory} (see also \S\ref{sec: homotopy invariance} above), we obtain the isomorphism $\Kgroup^G(X) \rightarrow \Kgroup^G(X')$ in the statement.

     For the moreover part, it is enough to observe that the composition 
     $$  B \GL_{\mathbf{n}} \longrightarrow B G_{\mathbf{w},\mathbf{n}}\longrightarrow B \GL_{\mathbf{n}} $$ 
     is the identity. The two maps here are those in Remark~\ref{rmk: natural_maps_G_Gln}.
\end{proof}

\subsubsection{Lambda Structure and Self-intersection Formula}

There exists a natural $\lambda$-structure on $\Kring_G(X)$, that is, a collection of maps $\lambda^i:\Kring_G(X)\rightarrow\Kring_G(X)$ with the property that $\lambda^i([E])=[\Lambda^iE]$ for every $G$-vector bundle $E$ over $X$. Here, $\Lambda^iE$ denotes the $i$-th exterior power. As in the whole paper, we are identifying vector bundles with their sheaves of sections, see Remark~\ref{rmk: vector bundles vs locally-free sheaves}. These maps satisfy the usual compatibility properties mimicking those of exterior powers.

\begin{Def}\label{def:euler_class}
    The $K$-theoretic Euler class of an element $\alpha\in\Kring_G(X)$ is defined by
\begin{equation*}
    \lambda_{-1}(\alpha):=\sum_{i=0}^{\infty}(-1)^i\lambda^i(\alpha).
\end{equation*}
\end{Def}
This is well-defined since $\lambda^i(\alpha)=0$ for $i$ large enough (depending on $\alpha$). The operator $\lambda_{-1}$ is multiplicative on exact sequences.

As the name suggests, the $K$-theoretic Euler class plays a role analogous to that of the top Chern class (also called the Euler class) in intersection theory. This is for instance illustrated by the self-intersection formula below.

\begin{Prop}[Self-intersection Formula, {\cite[Theorem 3.1]{Thomason-blowup}}]\label{prop: self intersection formula}
    Let $i:Z\hookrightarrow X$ be a regular embedding of smooth $G$-varieties, and let $N_{Z/X}$ be the normal bundle. Then, for every class $\alpha\in\Kring_G(X)$ we have
    \[
        i^*i_*(\alpha)=\lambda_{-1}(N_{Z/X}^{\vee})\cdot\alpha.
    \]
\end{Prop}

This is well-known \cite[Theorem 3.1]{Thomason-blowup} and a special case of the more general excess intersection formula in Theorem~\ref{prop: excess intersection formula}.

\subsubsection{The K\"unneth Property for $B\Gm$}

We show that $B\Gm$ satisfies the K\"unneth formula in K-theory; for Chow rings, this is already known~\cite[Lemma 10]{Oesinghaus18}.

\begin{Lemma}\label{lem: kunneth formula BGm}
    Consider the trivial $\Gm$-action on $X$, and let $t$ be the class of the geometric standard $\Gm$-representation. Then,
    \[
        \Kgroup([X/G]\times B\Gm)\simeq\Kgroup^{G\times\Gm}(X)\simeq\Kgroup^G(X)\otimes R(\Gm)
    \]
    as $\Kring_G(X)\otimes R(\Gm)$-modules. If $X$ is smooth, this induces an isomorphism of rings
    \[
        \Kring([X/G]\times B\Gm)\simeq\Kring_{G\times\Gm}(X)\simeq\Kring_G(X)[t,t^{-1}].
    \]
\end{Lemma}
\begin{proof}
    Let $\phi:\Kgroup^G(X)\otimes R(\Gm)\simeq\Kgroup^G(X)\otimes\Z[t,t^{-1}]\rightarrow\Kgroup^{G\times\Gm}(X)$ be the natural homomorphism of $\Kring_{G}(X)\otimes R(\Gm)$-modules. To show that $\phi$ is injective, it is enough to consider the composition of it with
    \[
        \Kgroup^{G\times\Gm}(X)\simeq\Kgroup^{G\times\Gm}(X\times\bA^n)\rightarrow\Kgroup^{G\times\Gm}(X\times(\bA^n\setminus0))\simeq\Kgroup^{G}(X\times\mathbb{P}^{n-1}),
    \]
    for every $n\geq0$, and apply the projective bundle formula of~\cite[Theorem 3.1]{Thomason-equivariantKtheory} (see also Theorem~\ref{prop: twisted weighted projection formula} below). Here, $G$ acts trivially on $\bA^n$, while $\Gm$ acts via the usual scaling action.
    To show that $\phi$ is surjective, note that every $(G\times\Gm)$-coherent sheaf over $X$ is simply a $G$-equivariant coherent sheaf $F$ with a $\Gm$-action commuting with that of $G$. Therefore, $F$ splits as a direct sum of $G$-equivariant coherent sheaves according to the $\Gm$-weights. Each summand can be written as the tensor product of the underlying $G$-equivariant coherent sheaf and a power of the standard presentation of $\Gm$. Thus, its class is in the image of $\phi$, hence the surjectivity.

    When $X$ is smooth, $\Kgroup([X/G]\times B\Gm)\simeq\Kring([X/G]\times B\Gm)$ is a ring and $\phi$ is a ring isomorphism.
\end{proof}

\subsubsection{Euler Class of Twisted Weighted Vector Bundles}\label{subsec: euler class}

In the same fashion as in~\cite{Arena-Obinna}, we now define and study the $\Gm$-equivariant Euler class of twisted weighted vector bundles $E$.

We start by considering the case of (untwisted) weighted vector bundles.
Given a $G$-equivariant vector bundle $E$ over $X$ and $a\in\Z$, we denote by $E^a$ the corresponding weighted vector bundle with weight $a$, and vice versa.

Let $E$ be a $G$-equivariant weighted vector bundle on $X$ with weights $0<w_1<\ldots<w_r$. Viewing $E$ as a $G$-equivariant vector bundle on $X\times B\Gm$ yields a $\Gm$-equivariant class $\lambda_{-1}^{\Gm}(E) \in \Kring_{\Gm \times G}(X)$, which is computed as follows. Writing $E=\oplus E_i^{a_i}$, where $E_i$ has rank $n_i$, then
\[
    \lambda_{-1}^{\Gm}(E)=\prod_i\lambda_{-1}^{\Gm}(E_i^{w_i})\in\Kring_{G\times\Gm}(X) \simeq \Kring_{G}(X \times B\Gm) \simeq \Kring_G(X)[t,t^{-1}].
\]
Note that we are writing $\lambda_{-1}^{\Gm}$ instead of $\lambda_{-1}$ to stress that we are keeping track of the $\Gm$-action on $E$, and the class lives in $\Kring_{G\times\Gm}(X)$ rather than $\Kring_{G}(X)$.

If $E_i$ is a line bundle, then $\lambda_{-1}^{\Gm}(E_i^{w_i})=1-[E_i^{w_i}]=1-[E_i]t^{w_i}$.
In general, the K-theoretic versions of the splitting principle~\cite[page 51]{fulton} and of~\cite[Proposition 3.6, Corollary 3.7]{Arena-Obinna} work with the same proofs, and yield the equality
\[
    \lambda_{-1}^{\Gm}(E_i^{w_i})=1-\lambda^1(E_i)t^{w_i}+\ldots+(-1)^{n_i}\lambda^{n_i}(E_i)t^{w_i\cdot n_i}.
\]
Finally, let $E \to X$ be a $G$-equivariant twisted weighted vector bundle on $X$ with weights $0<w_1<\cdots<w_r$. Recall from Remark~\ref{rmk: mistake in Arena-Obinna} that we canonically associate to $E$ a sequence of weighted vector bundles $E_i^{w_i}$ on $X$ (see also~\cite[\S2.1.9]{quek-rydh-weighted-blowup} and~\cite[Proposition~3.8]{Arena-Obinna}).

\begin{Def}\label{def: euler class twisted weighted vector bundle}
    We define
    \[
        \lambda_{-1}^{\Gm}(E):=\prod_i\lambda_{-1}^{\Gm}(E_i^{w_i})\in\Kring_G(X)[t,t^{-1}].
    \]
\end{Def}

\subsubsection{Twisted Weighted Projective Bundle Formula}\label{subsec: projective bundle formula}

Given a $G$-equivariant twisted weighted vector bundle $\pi:E\rightarrow X$ with 0-section $E_0\subset E$ and weights $0<w_1<\ldots<w_r$, let $\cP(E)=[(E\setminus E_0)/\Gm]$ be the associated projective bundle. Using the homotopy invariance for twisted weighted vector bundles, the projection and self-intersection formulas in the same way as~\cite[Lemma 2.5, Corollary 3.11]{Arena-Obinna}, we obtain:

\begin{Thm}\label{prop: twisted weighted projection formula}
    Let $E$ be a $G$-equivariant twisted weighted vector bundle over $X$, and let $E_i$ be the associated (untwisted) weighted vector bundles as in Subsection~\ref{subsec: euler class}. Set 
    $$
    \lambda_{-1}^{\Gm}(E^{\vee}):=\prod_i\lambda_{-1}^{\Gm}(E_i^{\vee}) \in \Kring_{G \times \Gm}(X).
    $$
    Then,
    \[
    \Kring_G(\cP(E))=\Kring_{G\times\Gm}(E\setminus E_0)\simeq\frac{\Kring_G(X)[t,t^{-1}]}{(\lambda_{-1}^{\Gm}(E^{\vee}))},
    \]
    where $t=[\sO_{\cP(E)}(1)]$ is the pullback of the geometric standard representation on $B\G_m$ along the natural morphism
    $
    \cP(E)\simeq[(E\setminus E_0)/\Gm]\longrightarrow B\Gm.
    $
\end{Thm}
\begin{proof}
    It remains to justify that $t$ corresponds to the class of $\sO_{\cP(E)}(1)$. Consider the map
$$
(E \smallsetminus E_0) \times \A^1 \longrightarrow (E \smallsetminus E_0) \times E,
\qquad
(v,\lambda) \longmapsto (v,\lambda v).
$$
This map is $\G_m$-equivariant if $\G_m$ acts anti-diagonally on $(E \smallsetminus E_0) \times \A^1$, while it acts on $(E \smallsetminus E_0) \times E$ by multiplication on the first factor only. Passing to the quotient by $\G_m$, the induced map identifies with the tautological inclusion
$$
\sO_{\cP(E)}(-1) \hookrightarrow E
$$
over $\cP(E)$. It follows that the pullback of the geometric standard representation from $B\G_m$ is identified with the dual of $\sO_{\cP(E)}(-1)$, namely $\sO_{\cP(E)}(1)$.
\end{proof}

This recovers~\cite[Proposition 2.5]{Edidin-Hu} in the simpler case of (untwisted unweighted) vector bundles.

\begin{Rmk}
    When $E=C_{I_\bullet}$, one simply has $\lambda_{-1}^{\Gm}(E^\vee)=\lambda_{-1}^{\Gm}(\sN_{\cX/\cY}^\vee)$.
\end{Rmk}

\begin{Example}\label{example M11}
    A simple but interesting example is that of the Deligne-Mumford stack of stable 1-pointed genus 1 curves $\overline{\cM}_{1,1}$. Indeed, in characteristic different from 2 and 3, $\overline{\cM}_{1,1}\simeq\cP(4,6)$, thus
    \[
        \Kring(\overline{\cM}_{1,1})\simeq\frac{\Z[t^{\pm1}]}{((1-t^{-4})(1-t^{-6}))},\qquad t=[\cO_{\cP(4,6)}(1)]
    \]
    by Theorem~\ref{prop: twisted weighted projection formula}. Explicitly, let $\pi:\cE\rightarrow\overline{\cM}_{1,1}$ be the universal curve, and let $\lambda\in\Kring(\overline{\cM}_{1,1})$ be the class of the Hodge line bundle $\pi_*\omega_{\cE/\overline{\cM}_{1,1}}$. Then, $t=\lambda$. To show this, recall the pullback of $\cE$ along the quotient map $\bA^2\setminus0\rightarrow\overline{\cM}_{1,1}$ coincides with the closed subscheme $Z\subset(\bA^2\setminus0)\times\bP^2$ defined as $Y^2Z=X^3+aXZ^2+bZ^3$, where $(a,b)$ are the coordinates on $\A^2\setminus0$, and $[X:Y:Z]$ those on $\bP^2$. The locus $Z$ is $\Gm$-invariant under the $\Gm$-action
    \[
    \begin{tikzcd}
        \lambda\cdot\left((a,b),[X:Y:Z]\right)\arrow[r,mapsto] & \Phi_\lambda((a,b),[X:Y:Z]) \colon= \left((\lambda^4a,\lambda^6b),[\lambda^2X:\lambda^3Y:Z]\right).
    \end{tikzcd}
    \]
    A global invariant section of the relative dualizing sheaf of $Z \to \mathbb{A}^2 \smallsetminus 0$ is $dx/2y$, which transforms with weight $-1$ under the $\Gm$-action, that is if $\phi_\lambda(x,y)=(\lambda^2 x,\lambda^3 y)$, then
    $$
    \phi_\lambda^*\left(\frac{dx}{2y}\right) = \lambda^{-1}  \frac{dx}{2y}.
    $$
    Because $\phi_\lambda^*$ is a map $(\pi_*  \Phi_\lambda^*\omega_{\mathcal{E}/ \overline{\cM}_{1,1}}) \to \pi_* \omega_{\mathcal{E}/ \overline{\cM}_{1,1}}$, it follows that $\pi_* \omega_{\mathcal{E}/ \overline{\cM}_{1,1}}$ is the pullback of the geometric standard representation from $B \Gm$, that is $\sO(1)$.
\end{Example}

\subsection{The K-theory of the Weighted Blowup of $\bA^d$ along the Origin}\label{sub:K-theory-of-local-weighted-blowup}

Recall that weighted blowups along a regular embeddings (Definition~\ref{def:weighted-blowup}) are smooth-locally obtained by flat pull-back from the weighted blowup of $\bA^d$ along the origin, with some (non-necessarily distinct) positive weights $\bold{w}:=(w_1,\dots, w_d)$. In order to understand the K-theory of a weighted blowup, it is then fundamental to compute the K-theory of $Bl^{\bold{w}}_0\bA^d$, which we do next.
\begin{Corollary} 
Let $Bl^{\bold{w}}_0\bA^d$ be the weighted blowup of the origin $0$ in $\bA^d$ described in \S\ref{sub:blowup-0-An}, with positive (non-necessarily distinct) weights $\bold{w}:=(w_1,\dots, w_d)$. Let $E$ be the weighted normal cone of the origin in $\A^d$. Then,  
    $$  \lambda_{-1}^{\Gm}(E^{\vee}) = \prod_{i=1,\dots,d}(1-t^{-w_i})\in\Kring_{\Gm}(0)
    $$
    where $t$ is the class of the geometric standard representation in $\Kring(B\G_m) \simeq \Kring_{\G_m}(0)$. Moreover, we have the following isomorphism of rings:
    \[
    \Kring(Bl_0^{\bold{w}}\A^d)\simeq\frac{\Z[t,t^{-1}]}{\prod_{i}(1-t^{-w_i})}.
    \]
    where $t=[\sO_{Bl_0^{\bold{w}}\A^d}( -\mathcal{P}(w_1,\ldots,w_d))]$.
\end{Corollary}
     
\begin{proof}
    The computation of  $\lambda_{-1}^{\Gm}(E)$ follows by observing that $E$ is the weighted vector bundle $\spec \Sym^{\bullet}(V)$ over a point, where $V=k^{\oplus d}$ has weights $(w_1,\dots,w_d)$. 
    
    The second isomorphism follows from the fact that, as explained in \S\ref{sub:blowup-0-An}, the weighted blowup is the total space of the line bundle $\sO_{\cP}(-1)$ over the exceptional divisor $\cP=\cP(w_1,\ldots,w_d)$, and from the formula for the K-theory of weighted projective bundles in Theorem~\ref{prop: twisted weighted projection formula}. We are only left to justify the equality $t=[\sO_{Bl_0^{\bold{w}}\A^d}(-\mathcal{P}(w_1,\ldots,w_d))]$. This follows from the fact that the pullback of the geometric standard representation to $\cP$ is $\sO_{\cP}(1)$ by Theorem~\ref{prop: twisted weighted projection formula}, together with the general fact that the pullback of a line bundle $L$ on a scheme $X$ along the projection $L \to X$ is canonically isomorphic to the dual of the ideal sheaf of the zero section of $X$ in $L$.
\end{proof}

\section{Equivariant Refined Pullbacks and Operational K-theory}\label{sec: gysin and operational}

In this section, we extend the results on $\Gm^n$-equivariant operational K-theory from \cite{AndersonPayne-OperationalKtheory}, to other groups.

\subsection{Refined Gysin morphism}\label{sub:refined-gysin}

Let $G$ be a smooth affine group scheme of finite type over $k$. Consider a cartesian square
\[
\begin{tikzcd}
X' \arrow[r, "f'"] \arrow[d, "g'"] & Y' \arrow[d, "g"] \\
X \arrow[r, "f"] & Y
\end{tikzcd}
\]
of $G$-equivariant morphisms. Because Tor-sheaves carry a canonical equivariant structure, the usual construction of refined Gysin morphism in K-theory works equivariantly. Given an $f$-perfect equivariant complex $\sP_\bullet$ on $X$, there is a pullback map $f^{\sP_\bullet}: \Kgroup^{G}(Y') \to  \Kgroup^{G}(X')$ defined by
\[
f^{\sP_\bullet}(\sF) = \sum_i (-1)^i \big[ \mathcal{T}\!or_i^Y(\sP_\bullet, \sF) \big].
\]

which only depends on the $K$-class of $\sP_\bullet$. When $f$ has finite Tor-dimension, we simply write 
$
f^! = f^{\sO_X}.
$
For example, this is the case when $f$ is flat or a regular embedding. Note in this case $f^!$ coincides with the derived pullback. If $f$ is perfect and $g=\Id$, we also write $f^*$ in place of $f^!$, as in this case $f^{!}$ is the usual pullback in $K$-theory. 

The next proposition summarizes the main properties of equivariant Gysin maps.

\begin{Prop}\label{prop: properties pullbacks}
\begin{enumerate}
    \item\label{point 1: prop: properties pullbacks} Consider the $G$-equivariant cartesian square
    \[
    \begin{tikzcd}[row sep=large, column sep=large]
    X'' \arrow[r] \arrow[d, "h'"'] & Y'' \arrow[d, "h"] \\
    X' \arrow[r, "f'"] \arrow[d, "g'"'] & Y' \arrow[d, "g"] \\
    X \arrow[r, "f"] & Y
    \end{tikzcd}
    \]
   where $h$ is proper and $f$ is either flat or a closed embedding.
    Let $\sP_\bullet$ be an equivariant $f$-perfect complex on $X$. Then
    \[
    f^{\sP_\bullet} \circ h_* = h'_* \circ f^{\sP_\bullet}.
    \]
    \item \label{point 2: prop: properties pullbacks}
    Next, consider the following $G$-equivariant diagram of cartesian squares:
    \[
    \begin{tikzcd}[row sep=large, column sep=large]
    X'' \arrow[r] \arrow[d] & Y'' \arrow[r] \arrow[d] & Z'' \arrow[d, "h"] \\
    X' \arrow[r] \arrow[d] & Y' \arrow[r] \arrow[d] & Z'  \\
    X \arrow[r,"f"] & Y  & 
    \end{tikzcd}
    \]
    where $f$ and $h$ are either flat or closed embeddings.  Let $\sP_\bullet$ be an equivariant $f$-perfect complex on $X$, and let $\sQ_\bullet$ be an equivariant $h$-perfect complex on $Z''$. 
    Then, for any class $\xi \in K_\circ^{G}(Y')$, we have
    \[
    f^{\sP_\bullet} \circ h^{\sQ_\bullet} (\xi)
    =
    h^{\sQ_\bullet} \circ f^{\sP_\bullet} (\xi)
    \in K_\circ^{G}(X'').
    \]
    \item\label{point 3: prop: properties pullbacks}
    Suppose \(f : X \to Y\) factors as a closed embedding 
    \(\iota: X \hookrightarrow M\) followed by a smooth projection \(p: M \to Y\). Let \(\sP_\bullet\) be an equivariant $f$-perfect complex on \(X\). 
    For any equivariant morphism \(Y' \to Y\), let $p': M \times_Y Y' \to Y'$ be the pullback of $p$. Then, we have 
    $
    f^{ \sP_\bullet} = \iota^{ \sP_\bullet} \circ (p')^*
    $
    as maps
    $
    K_\circ^{G}(Y') \to K_\circ^{G}(X').
    $
\end{enumerate}
\end{Prop}
\begin{proof}
    Identical to the proof of \cite[Lemmas 3.1, 3.2 and 3.3]{AndersonPayne-OperationalKtheory}.
\end{proof}

\begin{Thm}[Excess Intersection Formula]\label{prop: excess intersection formula}
    Suppose given a $G$-equivariant cartesian square of separated schemes of finite type over $k$ 
    \[
    \begin{tikzcd}
        X'\arrow[r,"i'"]\arrow[d,"f'"] & Y'\arrow[d,"f"]\\
        X\arrow[r,"i"] & Y
    \end{tikzcd}
    \]
    where $i$ and $i'$ are $G$-equivariant regular immersions with associate normal bundles $N$ and $N'$ respectively. Consider the induced exact sequence of $G$-equivariant bundles over $X'$
    \[
    \begin{tikzcd}
        0\arrow[r] & E^{\vee} \arrow[r] & f'^*N^{\vee}\arrow[r] & N'^{\vee}\arrow[r] & 0
    \end{tikzcd}
    \]
    where $E$ is the excess intersection bundle. Then, for every $G$-equivariant morphism $g:Y'' \rightarrow Y'$ and $\alpha\in\Kgroup^G(Y'')$ we have
    \begin{equation}
       i^! (\alpha) = (i')^!(\alpha) \cdot \lambda_{-1}((g')^{*}E^{\vee})
    \end{equation}
    in $\Kgroup^G(X'')$ where $g':X'':=Y''\times_{Y} X \simeq Y''\times_{Y'} X' \rightarrow X'$ is the natural projection.
\end{Thm}

\begin{proof}
    It is well known that there exists a natural isomorphism
    \[
        \mathrm{Tor}_k^{Y}(\sO_{Y'},\sO_X)\simeq\Lambda^k E^{\vee}
    \]
    as $\sO_{X'}$-modules by computing the $\mathrm{Tor}$-functor of $- \otimes_{\sO_{Y}} \sO_{Y'}$ using the Koszul complex associated to the regular embedding $i$, see \cite[Lemma 3.2]{Thomason-blowup} for a more detailed explanation. Moreover, this isomorphism is $G$-equivariant, by naturality. Without loss of generality, we assume $\alpha=[M]$ for $M$ a coherent sheaf on $Y''$.  Consider now the two functors
    \begin{itemize}
        \item $a(-):=- \otimes_{\sO_Y} \sO_{Y'}: \{ \sO_Y- {\rm modules}\} \rightarrow  \{ \sO_{Y'}-{\rm modules}\}  $;
        \item $b(-):=- \otimes_{\sO_{Y'}} M: \{ \sO_{Y'}- {\rm modules}\} \rightarrow  \{ \sO_{Y''}- {\rm modules}\} $;
    \end{itemize}
      and the convergent Grothendieck-Serre spectral sequence associated to the composition $b\circ a$ 
    $$ \operatorname{Tor}_j^{Y'}(M, \operatorname{Tor}_k^{Y}(\sO_{Y'},\sO_X)) \implies \operatorname{Tor}_{j+k}^Y(M,\sO_X). $$
    Since $ \mathrm{Tor}_k^{\sO_Y}(\sO_{Y'},\sO_X)$ is locally free as $\sO_{X'}$-module, we get
    $$  \operatorname{Tor}_j^{Y'}(M, \operatorname{Tor}_k^{Y}(\sO_{Y'},\sO_X)) =  \operatorname{Tor}_j^{Y'}(M, \sO_{X'}) \otimes_{\sO_{X'}} \operatorname{Tor}_k^{Y}(\sO_{Y'},\sO_X))$$
    which implies that
    $$ \operatorname{Tot}\left( \operatorname{Tor}_{\bullet}^{Y'}(M, \operatorname{Tor}_{\bullet}^{Y}(\sO_{Y'},\sO_X)\right)\simeq \operatorname{Tor}_{\bullet}^{Y'}(M, \sO_{X'}) \otimes_{\sO_{X'}} \operatorname{Tor}_{\bullet}^{Y}(\sO_{Y'},\sO_X)$$
    where $\operatorname{Tot}$ denotes the total complex.
    Applying the Euler characteristic on $X''$ to the isomorphism above, we get 
    $$\chi\left(\operatorname{Tot}\left(\operatorname{Tor}_{\bullet}^{Y'}\left(M, \operatorname{Tor}_{\bullet}^{Y}(\sO_{Y'},\sO_X)\right)\right)\right) =\chi(\operatorname{Tor}_{\bullet}^{Y'}(M, \sO_{X'})) \cdot\chi(\operatorname{Tor}_{\bullet}^{Y}(\sO_{Y'},\sO_X))$$ 
    whereas the convergence of the spectral sequence gives us that
    $$ \chi\left(\operatorname{Tot}\left(\operatorname{Tor}_{\bullet}^{Y'}\left(M, \operatorname{Tor}_{\bullet}^{Y}(\sO_{Y'},\sO_X)\right)\right)\right)= \chi \left( \operatorname{Tor}_{\bullet}^Y(M,\sO_X) \right).$$
    The statement follows by noticing that 
    \begin{itemize}
        \item by definition, $i^!([M])=  \chi \left( \operatorname{Tor}_{\bullet}^Y(M,\sO_X) \right)$;
        \item by definition, $(i')^!([M])=  \chi \left( \operatorname{Tor}_{\bullet}^{Y'}(M,\sO_{X'}) \right)$;
        \item $\lambda_{-1}(E^{\vee})=\sum (-1)^k \Lambda^k E^{\vee} = \chi(\operatorname{Tor}_{\bullet}^{Y}(\sO_{Y'},\sO_X))$.
    \end{itemize}
\end{proof}

The self-intersection formula (Theorem~\ref{prop: self intersection formula}) follows from the above theorem taking $Y'=X$. 

A version of Theorem~\ref{prop: excess intersection formula} already appears in \cite[Theorem 3.1]{Thomason-blowup}. Assuming that the map $f$ in the statement is lci, the excess intersection formula in Theorem~\ref{prop: excess intersection formula} implies \cite[Theorem 3.1]{Thomason-blowup}.

We end this subsection with some results which will turn out to be useful later on.

\begin{Lemma}\label{lem:lci-factorization}
    Let $f:\cX \rightarrow \cY$ be a morphism of smooth algebraic stacks and suppose that $\cY$ is the quotient of a smooth separated algebraic space $Y$ by $G$. Then $f$ is the composition of a regular closed embedding followed by a smooth morphism.
\end{Lemma}

If $\cY$ is a separated smooth algebraic space, then this follows by the usual graph argument, that is considering the decomposition of $f$ given by the graph of $f$. The hypothesis $\cY$ separated is used to ensure that the graph map is a closed embedding. If $\cY$ is not separated we proceed as follows.

\begin{proof}

We can consider the pull-back of $f$ through the atlas $Y\rightarrow \cY$, which we denote by $f_Y:\cX_Y:=\cX \times_{\cY} Y\rightarrow Y$. Considering the graph factorization of $f_Y$, we get the composition 
 $$ \cX_Y \longrightarrow \cX_Y \times Y \longrightarrow Y$$
 where the first morphism is a regular closed embedding (because $Y$ is a separated smooth algebraic space) and the second one is a smooth morphism (since $G$, and thus $\cX_Y$, is smooth). A straightforward verification shows that if we endow $\cX_Y\times Y$ with the diagonal action of $G$, then the graph of $f_Y$ is $G$-equivariant (since $f_Y$ is). By taking the quotient by $G$ (see for instance Theorem 4.1 of \cite{Rom05}) of the composition, we get the desired decomposition of $f$, since being a regular closed embedding and being smooth are properties which are smooth-local on the target.
\end{proof}

Before going the next result, let us recall a classical piece of notation. Given $f:\cX \rightarrow \cY$ a morphism of pure dimensional algebraic stacks of finite type over a field, we denote by $d(f)$ the quantity $\dim \cX - \dim \cY$.
\begin{Cor}\label{cor:restriction of Gysin}
    Suppose given a cartesian square of connected smooth algebraic stacks 
    \[
    \begin{tikzcd}
        \cX'\arrow[r,"g'"]\arrow[d,"f'"] & \cX\arrow[d,"f"]\\
        \cY'\arrow[r,"g"] & \cY
    \end{tikzcd}
    \]
    such that the relative dimension of $f$ and $f'$ are the same, i.e. $d(f)=d(f')$. Suppose moreover that $\cY$ is a quotient of a smooth separated algebraic space $Y$ by $G$. Then for every $\cY''\rightarrow\cY'$ we have that $f^!=(f')^!$ as morphisms $\Kgroup(\cY'') \rightarrow \Kgroup(\cX'')$ where $\cX''\simeq \cX\times_{\cY} \cY'' \simeq \cX' \times_{\cY'} \cY''$.
\end{Cor}

\begin{proof}
 First of all, we can use Lemma~\ref{lem:lci-factorization} to decompose $f$ as a regular closed embedding $i$ followed by a smooth morphism $p$ and let us denote by $i'$ and $p'$ the pull-back of $i$ and $p$ respectively through $g$. Thanks to the smoothness of $p$, we get that $d(p)=d
 (p')$ and consequently $d(i)=d(i')$. Moreover, $i'$ is a closed embedding between smooth algebraic stacks, thus it is a regular closed embedding. By Proposition~\ref{prop: properties pullbacks}\eqref{point 3: prop: properties pullbacks}, we get that it is enough to show the result for $i$ and $i'$. Finally, this follows from Theorem~\ref{prop: excess intersection formula}, since $d(i)=d(i')$ implies that the two closed embedding have the same codimension, i.e. the excess intersection bundle is trivial.
\end{proof}

\subsection{$G$-equivariant operational K-theory}\label{sub:operational-K-theory}
We start by briefly recalling the definitions.

\begin{Def}\label{def: operational K-theory}
    Let $f \colon X \to Y$ be a $G$-equivariant morphism of schemes. For each equivariant morphism $g \colon Y' \to Y$, form the cartesian square
\begin{equation}\label{eqn: cartesian diagram}
\begin{tikzcd}
X' \arrow[r, "f'"] \arrow[d, "g'"] & Y' \arrow[d, "g"] \\
X \arrow[r, "f"] & Y
\end{tikzcd}
\end{equation}
An element of the operational equivariant $K$-theory $\mathrm{op}\Kring_{G}(f:X \to Y)$ consists of operators
\[
c_g \colon \Kgroup^{G}(Y') \to \Kgroup^{G}(X')
\]
for all $g \colon Y' \to Y$, which commute with proper pushforward and refined Gysin maps; that is, they satisfy properties (A1) and (A2) in \cite[Definition 4.1]{AndersonPayne-OperationalKtheory}, with $T$ replaced by $G$. 
\end{Def}

\begin{notation}
    We will omit $f$ and simply write $\opK_{G}(X \to Y)$ when there is no risk of confusion. We write $\mathrm{op}\Kring_{G}(X)$ for the operational $K$-group of the identity morphism on $X$.
\end{notation}
Note that $\mathrm{op}\Kring_{G}(X)$ is an associative ring with unit under composition of endomorphisms, and it is contravariant in $X$.

The $G$-equivariant operational $K$-groups admit various operations and satisfy properties similar to those in (O1)--(O3) and (P1)--(P7) in \cite[\S4.1. and \S4.2]{AndersonPayne-OperationalKtheory}, with $T$ replaced by $G$.

By Proposition~\ref{prop: properties pullbacks}, if $f \colon X \to Y$ is a $G$-equivariant flat morphism or a regular embedding, then $f^!$ defines an element of $\opK_{G}(X \to Y)$, called the \emph{canonical orientation} of $f$ and denoted by $[f]$. 

In general, given $f:X \to Y$ and $g: Y \to Z$ of finite Tor-dimension, it is not clear whether $f^! \circ g^! = (g \circ f)^!$. However, by Proposition~\ref{prop: properties pullbacks}\eqref{point 3: prop: properties pullbacks}, this equality holds when $f$ is a regular embedding and $g$ is a smooth morphism.

\begin{Prop}\label{prop: composition with smooth is iso in Ktheory}
Let $X \to Y$ be an arbitrary $G$-equivariant morphism, and let $g \colon Y \to Z$ be smooth and equivariant. Then,
\[
\cdot [g] \colon \opK_{G}(X \to Y) \to \opK_{G}(X \to Z)
\]
is an isomorphism.
\end{Prop}

\begin{proof}
    Identical to that of \cite[Proposition 4.3]{AndersonPayne-OperationalKtheory}.
\end{proof}

\subsubsection{Isomorphism between operational and non-operational K-theories}

In this section, we will need to add further assumptions on the group $G$. We will assume that $G$ is a split reductive group over the field $k$, with simply connected commutator subgroup. Let $\Gm^n \subseteq G$ be a maximal torus.

The following is one of the main results of this subsection.

\begin{Prop}\label{prop: iso over GLn}
    The map $\opK_{G}(X \to \mathsf{pt}) \to \Kgroup^{G}(X)$ taking $c$ to $c_{\mathrm{Id}_{\mathsf{pt}}}([\sO_{\mathsf{pt}}])$ is an isomorphism of groups.
\end{Prop}

When $G$ is a split torus, this is proved in \cite[Lemma 2.2]{AndersonPayne-OperationalKtheory} and uses the fact that, for a scheme $X$ with a $\Gm^n$-action, the $K$-theory groups $\Kgroup^{\Gm^n}(X)$ are additively generated by the classes $[\sO_V]$, where $V$ ranges over the $\Gm^n$-equivariant subvarieties of $X$.
We do not have such a result for $G$-equivariant schemes, but we will deduce Proposition~\ref{prop: iso over GLn} by reducing to the torus equivariant situation. We defer the proof of Proposition \ref{prop: iso over GLn} until after the following discussion.

Note that if $X$ is any scheme with a $\Gm^n$-action, there is a natural isomorphism
\[
[X/\Gm^n] \simeq \bigg[ \frac{(X \times G)/\Gm^n}{G} \bigg],
\]
and thus $\Kgroup^{\Gm^n}(X)$ is naturally identified with 
$
\Kgroup^{G}\big((X \times G)/\Gm^n\big).
$
Note that $(X \times G)/\G^m_n$ is scheme by \cite[Proposition 23]{EG}.
Here $\Gm^n$ acts anti-diagonally on $X \times G$ and $G$ acts on $X \times G$ by multiplication on $G$.

If the $\Gm^n$-action on $X$ is the restriction of a $G$-action, then the action map induces a natural smooth $G$-equivariant map
\[
(X \times G)/\Gm^n \longrightarrow X,
\]
which modulo $G$ can be naturally identified with $[X/\Gm^n] \to [X/G]$. In particular, the induced map
\begin{equation}\label{eqn: inclusion in TK}
\Kgroup^{G}(X) \longrightarrow 
\Kgroup^{G}\big((X \times G)/\Gm^n\big) \simeq \Kgroup^{\Gm^n}(X), \quad
\alpha \longmapsto \alpha^{\Gm^n}
\end{equation}
is injective by \cite[Proposition 31]{Merkurjev-equivariantKtheory}. This is the only point where the assumptions on $G$ in Proposition~\ref{prop: iso over GLn} are used.

Finally, given a $\Gm^n$-equivariant cartesian square \begin{equation*}
\begin{tikzcd}
X' \arrow[r, "f'"] \arrow[d, "g'"] & Y' \arrow[d, "g"] \\
X \arrow[r, "f"] & Y
\end{tikzcd}
\end{equation*}
where $X$ and $Y$ are further endowed with $G$-actions and $f$ is $G$-equivariant, one obtains an associated $G$-equivariant square diagram 
\[
\begin{tikzcd}
(X' \times G)/\Gm^n \arrow[r, "(f')^{G}"] \arrow[d, "(g')^{G}"] & (Y' \times G)/\Gm^n \arrow[d, "g^{G}"] \\
X \arrow[r, "f"] & Y
\end{tikzcd}
\]

The vertical map $g^G$ is induced by the composition
\[
\begin{tikzcd}
    Y' \times G \arrow[r,"f \times \mathrm{Id}"] & Y \times G \arrow[r] & Y    
\end{tikzcd}
\]
where the second map is the action map. The morphism $(g')^G$ is defined analogously.

\begin{Lemma}\label{lemma: T operational class}
Let $c \in \opK_{G}(f \colon X \to Y)$. For any $\Gm^n$-equivariant morphism
$
g \colon Y' \to Y,
$
set $X' = X \times_Y Y'$ and identify $\Kgroup^{\G_m^n}(Y')$ and $\Kgroup^{\G_m^n}(X')$ with $\Kgroup^{G}((Y' \times G)/\Gm^n)$ and $\Kgroup^{G}((X' \times G)/\Gm^n)$ respectively, as explained above. Then,
\[ 
c_g^{\Gm^n} \colon 
\Kgroup^{G}\big((Y' \times G)/\Gm^n\big) 
\longrightarrow 
\Kgroup^{G}\big((X' \times G)/\Gm^n\big), 
\quad
\alpha \longmapsto c_{g^{G}}(\alpha).
\]
defines a $\Gm^n$-equivariant operational class $ c^{\Gm^n} \in \opK_{\Gm^n}(f: X \to Y). $

Moreover, the assignment $c \mapsto c^{\Gm^n}$ is a group homomorphism
\[
\opK_{G}(X \to Y) \longrightarrow \opK_{\Gm^n}(X \to Y).
\]
\end{Lemma}

\begin{proof}
    The commutativity with proper pushforwards and Gysin pullbacks for $c^{\Gm^n}$ follows immediately from that for $c$.
\end{proof}

\begin{Lemma}\label{lemma: compatibility T vs GL}
    With the notation of Lemma \ref{lemma: T operational class}, suppose further that $g$ is a $G$-equivariant map. For a class $\alpha \in \Kgroup^{G}(Y')$, one has
    $$
    c_g^{\Gm^n}(\alpha^{\Gm^n})= \big( c_g(\alpha) \big)^{\Gm^n}
    $$
    where $(-)^{\Gm^n}$ denotes the associated classes in $\Kgroup^{\Gm^n}$ under the inclusion in Equation~\eqref{eqn: inclusion in TK}.
\end{Lemma}

\begin{proof}
Consider the $G$-equivariant cartesian diagram
    \[
        \begin{tikzcd}[row sep=large, column sep=large]
        ((X \times_Y Y') \times G)/\Gm^n \arrow[r] \arrow[d, "h'"'] & (Y' \times G)/\Gm^n \arrow[d, "h"] \\
        X\times_Y Y' \arrow[r, "f'"] \arrow[d, "g'"'] & Y' \arrow[d, "g"] \\
        X \arrow[r, "f"] & Y
        \end{tikzcd}
    \]
We have
$$
c_g^{\Gm^n}(\alpha^{\Gm^n})=c_{g \circ h}( h^{!}(\alpha))= h^{!}(c_g(\alpha))= (h')^{!}(c_g(\alpha))=\big( c_g(\alpha) \big)^{\Gm^n}.
$$
The first and last equalities are by the definition of $(-)^{\Gm^n}$, while the second follows from (A2) in Definition~\ref{def: operational K-theory} (note that $h$ is flat) and the third holds because $h$ is a flat morphism.
\end{proof}

\begin{proof}[Proof of Proposition \ref{prop: iso over GLn}]

We can explicitly describe the inverse map as follows. Define

$$
\Kgroup^{G}(X) \to \opK_{G}(X \to \mathsf{pt}), \qquad \mathcal{F} \mapsto c^{\mathcal{F}}
$$
where for all $G$-equivariant schemes $Y$ and classes $\alpha \in \Kgroup^{G}(Y)$ we set  
$$
c^{\sF}(\alpha)=\sF \boxtimes \alpha \in \Kgroup^{G}(X \times Y).
$$

Clearly, $c^{\sF}([\sO_{\mathsf{pt}}]) = \sF$, so it remains to check that the other composition is the identity, namely that
\[
c_{\mathrm{Id}_{\mathsf{pt}}}([\sO_{\mathsf{pt}}]) \boxtimes \alpha
= c_g(\alpha) \in \Kgroup^{G}(X \times Y),
\]
where $g \colon Y \to \mathsf{pt}$ is the structural morphism and $\alpha \in \Kgroup^{G}(Y)$. Since the map~\eqref{eqn: inclusion in TK} is injective, it suffices to verify the equality after passing to $\Kgroup^{\Gm^n}(X \times Y)$. By Lemma~\ref{lemma: compatibility T vs GL}, the statement reduces to the $\Gm^n$-equivariant setting, where it follows from \cite[Lemma~2.2]{AndersonPayne-OperationalKtheory}.
\end{proof}

\begin{Corollary}\label{cor: opK=K}
    Suppose that $X$ is smooth and that the commutator subgroup of G is simply connected. Then the canonical map
    $$
    \opK_{G}(X)\longrightarrow \Kgroup^{G}(X)
    $$
    taking $c$ to $c_{\mathrm{Id}_X}([\sO_X])$ is an isomorphism of rings.
\end{Corollary}

\begin{proof}
    Combine Propositions \ref{prop: composition with smooth is iso in Ktheory} and \ref{prop: iso over GLn}.
\end{proof}

\section{K-Theory of a Weighted Blowup}\label{sec:K-theory-weighted-blowups}

\subsection{Gysin for Weighted Blowups}\label{sub: Gysin for Weighted Blowups}

As the title anticipates, in this subsection we describe how to compute the Gysin morphism for weighted blowups.

Let $\cX \subseteq \cY$ be a weighted regular closed embedding, and let $f: \widetilde{\cY} \to \cY$ denote the associated weighted blowup. Denote by $\tcX$ the reduced exceptional divisor over $\cX$, and form the commutative blowup diagram
\[
\begin{tikzcd}
    \tcX\arrow[r,"j"]\arrow[d,"g"] & \tcY\arrow[d,"f"]\\
    \cX\arrow[r,"i"] & \cY
\end{tikzcd}
\]
We wish to define an element $f^!$ that serves as the Gysin pullback for the map $f$. Ideally, we would want this element to live in the operational K-theory group:
\[
    f^! \in \opK(\widetilde{\cY} \to \cY)
\]
However, because we have not quite defined an operational K-theory $\opK(\widetilde{\cY} \to \cY)$, we instead consider
\[
\hat{f}^! \in \opK_{\GL_n \times \Gm}\bigl(D_{X_\bullet}Y \smallsetminus (X \times \bA^1) \to Y\bigr).
\]
Here, the $\GL_n \times \Gm$-equivariant lci map $\hat{f}$ is the composition 
$$
D_{X_\bullet}Y \smallsetminus (X \times \bA^1) \longrightarrow \widetilde{Y} \longrightarrow Y,
$$
$\Gm$ acts trivially on $Y$, and the $\GL_n \times \Gm$-action on 
$D_{X_\bullet}Y \smallsetminus (X \times \bA^1)$ is that given by the identification
\begin{equation}\label{eqn: GmXGLn-action-def-normal-cone}
D_{X_\bullet}Y \smallsetminus (X \times \bA^1) 
\simeq 
\bigl( D_{\cX_\bullet}\cY \smallsetminus (\cX \times \bA^1) \bigr) 
\times_{\widetilde{\cY}} \widetilde{Y},
\end{equation}
together with the $\Gm$-action on 
$D_{\cX_\bullet}\cY \smallsetminus (\cX \times \bA^1)$ and the $\GL_n$-action on $\widetilde{Y}$.
Given the commutative diagram
\begin{equation}\label{eqn: diagram cone}
\begin{tikzcd}[row sep=large, column sep=large]
C_{I_{\bullet}} \smallsetminus \operatorname{Im}(\sigma_{I_{\bullet}})
\arrow[r] \arrow[d] 
& D_{X_\bullet}Y \smallsetminus (X \times \bA^1) \arrow[d] \\
\widetilde{X}= \Proj_{X}(\sC_{I_{\bullet}})
\arrow[r, "i'"] \arrow[d, "f'"'] 
& \widetilde{Y} \simeq \left[\left( D_{X_\bullet}Y \smallsetminus (X \times \bA^1)\right)/\Gm\right] 
\arrow[d, "f"] \\
X \arrow[r, "i"] & Y
\end{tikzcd}
\end{equation}
where we endow both $X$ and $Y$ with the trivial $\Gm$-action, we obtain a Gysin pullback, which we denote by $f^!$,
\[
f^! : \Kgroup(\cX) \longrightarrow \Kgroup(\widetilde{\cX}),
\]
defined as the composition
\[
\begin{aligned}
\Kgroup(\cX) 
&\xrightarrow{\mathrm{Id} \otimes 1} 
\Kgroup(\cX) \otimes R(\Gm)
\simeq \Kgroup(\cX \times B\Gm)
\simeq \Kgroup^{\GL_n}(X \times B\Gm) \\
&\simeq \Kgroup^{\GL_n \times \Gm}(X)
\xlongrightarrow{\hat{f}^!} 
\Kgroup^{\GL_n \times \Gm}\bigl(C_{I_{\bullet}} \smallsetminus \operatorname{Im}(\sigma_{I_{\bullet}})\bigr)
\simeq \Kgroup(\widetilde{\cX}).
\end{aligned}
\]

Note that the diagram \eqref{eqn: diagram cone} is not cartesian, but $\Kgroup^{\GL_n \times \Gm}$ is insensitive to passing to the underlying reduced subscheme by the discussion in \S\ref{sec:excision}.

\begin{Prop}\label{prop:existence of gamma}
    There exists a class $\gamma \in \Kgroup(\widetilde{\cX})$ such that $f^!(\alpha)= g^*(\alpha) \cdot \gamma \in \Kgroup(\widetilde{\cX})$ for all $\alpha \in \Kgroup(\cX)$. Moreover, we have that $\gamma=f^{!}(1)$.
\end{Prop}
\begin{proof}
    By Proposition \ref{prop: composition with smooth is iso in Ktheory}, we have an isomorphism 
    $$
    \cdot [g] \colon \opK_{\GL_n \times \Gm}(C_{I_{\bullet}} \smallsetminus \operatorname{Im}(\sigma_{I_{\bullet}})) \longrightarrow \opK_{\GL_n \times \Gm}(C_{I_{\bullet}} \smallsetminus \operatorname{Im}(\sigma_{I_{\bullet}}) \to X).
    $$
    In particular, there exists
    $$
    \gamma \in \Kring(\widetilde{\cX}) \simeq \Kgroup(\widetilde{\cX}) \simeq \Kgroup^{\GL_n \times \Gm}( C_{I_{\bullet}} \smallsetminus \operatorname{Im}(\sigma_{I_{\bullet}})) \simeq \opK_{\GL_n \times \Gm}(C_{I_{\bullet}} \smallsetminus \operatorname{Im}(\sigma_{I_{\bullet}}))
    $$
    such that $g^*(-) \cdot \gamma= f^! \in \opK_{\GL_n \times \Gm}(C_{I_{\bullet}} \smallsetminus \operatorname{Im}(\sigma_{I_{\bullet}}) \to X)$. The last isomorphism is Corollary \ref{cor: opK=K} and $i : X \hookrightarrow Y$ denotes the inclusion. This concludes.
\end{proof}

Let $E$ be the weighted conormal cone of $\cX \subseteq \cY$. We recall that in Theorem~\ref{prop: twisted weighted projection formula}, we defined
$$ \lambda^{\Gm}_{-1}(E^{\vee}) \in \Kring(\cX \times B\Gm)$$
(equal to the Euler class of the dual of the normal bundle). Since $\Kring(\cX \times B\Gm)\simeq \Kring(\cX)[t,t^{-1}]$ where $t$ is the {class of the} standard geometric $\Gm$-representation, we have that $ \lambda^{\Gm}_{-1}(E^{\vee}) = p(t^{-1})$ for some polynomial $p$ of degree $w$ (where $w$ here is the sum of the $\Gm$-weights)  with coefficients in $\Kring(\cX)$. The polynomial will be useful to describe the class $\gamma$ of Proposition~\ref{prop: composition with smooth is iso in Ktheory}.

Let us start by computing the element $\gamma$ in a special case.

\begin{Rmk}
    The construction of $Bl_0^{\bf{w}} \bA^d$ in \S\ref{sub:blowup-0-An} is $\G_m^d$-equivariant, therefore we obtain a natural action of $\Gm^d$ on $Bl_{0}^{w}\A^d$ and an isomorphism 
    $$ [Bl_{0}^w\A^d/\Gm^d] \simeq Bl_{B\G_m^d}^w[\A^d/\G_m^d]$$
    where $B\G_m^d\subset [\A^d/\G_m^d]$ is the zero section.
\end{Rmk}

\begin{Lemma}\label{lemma: T-equivariant case}
    Let
    \[
    \begin{tikzcd}
        F:{[Bl_0^{\bf{w}} \bA^d/\G_m^d]} \arrow[r] & {[\bA^d/\G_m^d]}
    \end{tikzcd}
    \]
    be the blowup of $[\bA^d/\G_m^d]$ at $[0/\G_m^d]$ with weights ${\bf{w}}=(w_1,\dots,w_d)$. Then 
    $$ F^{!}(1) = \frac{ p(t^{-1})-p(1)}{ t-1} \in \Kring([\cP(w_1,\dots,w_d)/\G_m^d])$$
     where $t=[\sO(1)]$.
\end{Lemma}

\begin{proof}
Applying \S\ref{sec: representation rings} and Theorem~\ref{prop: twisted weighted projection formula}, we have the following isomorphisms
$$\Kring([\bA^d/\Gm^d])\simeq \Kring(B\Gm^d) \simeq \bZ[t_1^{\pm 1}, \dots t_d^{\pm 1}]; \quad \Kring([\cP(w_1, \dots w_d)/\Gm^d]) \simeq \frac{\bZ[t_1^{\pm 1}, \dots t_d^{\pm 1},t^{\pm 1}]}{p(t^{-1})}$$ where  $p(t^{-1})=\lambda_{-1}^{\Gm}([\bA^d/\Gm^d]^\vee)=\prod_{i=1}^d(1-t_i^{-1}t^{-w_i})$.

Let us set the notation:
$$
\begin{tikzcd}
\left[\cP(w_1,\dots,w_d)/\Gm^d\right] \arrow[r, "J"] \arrow[d, "G"] &  \left[Bl_0^{\bf{w}} \bA^d/\Gm^d\right] \arrow[d, "F"] \\
\left[0/\Gm^d\right] \arrow[r, "I"] & \left[\bA^d/\Gm^d\right]
\end{tikzcd}
$$ 
By Proposition~\ref{prop: properties pullbacks} we have $J_*F^!(1)=F^*I_*(1)$ and by applying $J^*$ to both sides and using the self-intersection formula in Theorem~\ref{prop: self intersection formula}, we obtain the equality
$$(1-t)\cdot F^!(1)=\lambda_{-1}(\sO_{\widetilde X}(-1)^\vee)\cdot F^!(1)=G^*(1) \cdot \lambda_{-1}(G^*\sN^\vee_{[0/T]}[\bA^d/T])=G^*(1) \cdot p(1).
$$

Now, $p(t^{-1})$ is the zero class in $\Kring_{\Gm^d}(\cP(w_1,\dots,w_d))$ so by adding $G^*(1) \cdot p(t^{-1})$ to the left hand side and dividing by $(1-t)$ both sides we obtain the desired equality.

The ring $A = \mathbb{Z}[t_1^{\pm 1}, \dots, t_d^{\pm 1}, t^{\pm 1}]$ is a UFD. Note that if $R$ is a UFD, and $f,g\in R\setminus0$, then $f$ is (either 0 or) a 0-divisor in $R/g$ iff g is (either 0 or) a 0-divisor in $R/f$. The quotient $A/(1-t)$ is a ring of Laurent polynomials and thus a UFD. The image of $p(t^{-1})$ in this ring is identified with $p(1)$ which is non-zero (and thus a non-zero divisor) in $\mathbb{Z}[t_1^{\pm 1}, \dots, t_d^{\pm 1}]$. It follows that $(1-t)$ is a non-zero divisor in $A/p(t^{-1})$.
\end{proof}

\begin{Lemma} \label{lemma: upper shriek of normal bundle}
     Let $\cX$, $\cY$, $\widetilde{\cX}$, $\widetilde{\cY}$ as defined in \S\ref{sec: notation-resolution}. Let $\cC:=\cC_{I_\bullet}$ the weighted normal cone  of $\cX$ in $\cY$ and $f^\cC: \widetilde{\cC} \to \cC$ be the weighted blowup of $\cX$ in $\cC$ induced by the weighted embedding of $\cX$ in $\cY$. Then the induced maps $f^!: \Kring(\cX) \to \Kring(\widetilde{\cX})$ and $(f^\cC)^!: \Kring(\cX) \to  \Kring(\widetilde{\cX})$ coincide.
\end{Lemma}

\begin{proof}
    The proof is similar to \cite[Theorem 5.2]{Arena-Obinna}. Consider the weighted blowup construction of the deformation to the normal cone explained in Remark~\ref{rem:blowup-of-def-cone}. More specifically, we have a weighted blowup diagram over $\A^1$
    $$
    \begin{tikzcd}
    \cZ \arrow[d] \arrow[r, hook]   & \widetilde{\cD} \arrow[d, "f_D"] \\
    \cX \times \A^1 \arrow[r, hook] & \cD               \end{tikzcd}
    $$
    where $\cD$ is the deformation to the normal cone. Recall also that $(f_D)_t=f$ if $t\neq 0$ and $(f_D)_0=f^{\cC}$. To conclude, it is enough to show $f_D^!=(f_D)_t^!$ for every $t \in \A^1$. Corollary~\ref{cor:restriction of Gysin} tells us that it is enough to show that $d(f_D)=d\left((f_D)_t\right)$ for every $t\in \A^1$. The dimensional equality is a simple consequence of the fact that $\widetilde{\cD}$ and $\cD$ are both flat over $\bA^1$.    
\end{proof}

\begin{Prop}\label{prop: f upper shriek}
     In the setting of Proposition~\ref{prop:existence of gamma}, we have the following identification
     $$ f^{!}(1) = \frac{ p(t^{-1}) - p(1)}{ t-1} \in \Kring(\widetilde{\cX})$$
     where $t=[\sO_{\widetilde{\cX}}(1)]$.
\end{Prop}

\begin{proof}
    Without loss of generality, we can assume that $\cY$ is a twisted weighted vector bundle over $\cX$ by Lemma~\ref{lemma: upper shriek of normal bundle}.
    Combining Proposition~\ref{prop:classical-splitting-principle} and Proposition~\ref{lemma: BT-toBGwn}, we find $\cX' \to \cX$ such that the induced pullback $\Kring(\cX) \to \Kring (\cX')$ is injective and the pullback $\cY' \to \cX'$  of $\cY \to \cX$ is a direct sum of line bundles over $\cX'$. Note that:
    \begin{enumerate}
        \item[$\bullet$] the map $\Kring(\widetilde{\cX}) \to \Kring (\widetilde{\cX}')$ induced by $\widetilde{\cX}'= \widetilde{\cX} \times_{\cY} \cY' \simeq \widetilde{\cX} \times_{\cX} \cX' \to \widetilde{\cX}$  is injective, and
        \item[$\bullet$] the pullback of $f^!(1) \in \Kring(\widetilde{\cX})$ to $\Kring(\widetilde{\cX}')$ is equal to $(f')^!(1)$ where $f': Bl_{\cX'}\cY' \to \cY'$.
    \end{enumerate}    
    We may thus replace $\cX$, $\cY$ with $\cX'$ and $\cY'$. Note also that $\cX'$ comes with a map to $B\Gm^d$ given by
    $$
    \cX' \simeq \cX \times_{B G_{\mathbf{a},\mathbf{w}}} B \Gm^d\longrightarrow B\Gm^d
    $$
    and $\cY'= [\mathbb{A}^d /\Gm^d] \times_{B \Gm^d} \cX'$. Thus, we have the following cartesian diagram.
$$
\begin{tikzcd}
\widetilde \cX' \arrow[rd, shift right] \arrow[dd] \arrow[rr, "\phi", shift right=2] &  & \left[\mathcal{P}(w_1,...,w_d)/\G_m^d\right] \arrow[rd, shift right] \arrow[dd] & \\
& \widetilde \cY' \arrow[rr, shift right=2] \arrow[dd, "f"] & & \left[Bl_0^{\bf{w}}\A^d/\G_m^d\right] \arrow[dd, "F"] \\
\cX' \arrow[rd, shift right] \arrow[rr, shift right=2] &  & \left[0/\G_m^d\right] \arrow[rd, shift right]   & \\ 
& \cY' \arrow[rr, shift right=2] &  & \left[\mathbb{A}^d/\G_m^d\right]
\end{tikzcd}
$$

By Proposition~\ref{prop: properties pullbacks} and Corollary~\ref{cor:restriction of Gysin}, we have that $$f^!(1)=\phi^*F^!(1)=\phi^*\left(\frac{\lambda_{-1}^{\Gm}([\bA^d/\G_m^d]^\vee)-\lambda_{-1}([\bA^d/\G_m^d]^\vee)}{\lambda_{-1}(\sO_{[\cP(w_1,\dots,w_d)/\G_m^d]}(1))}\right)=\frac{ p(t^{-1}) - p(1)}{ t-1} $$ concluding the proof.
\end{proof}

\subsection{The Blowup exact Sequence}\label{sub:blowup-main-theorem}

Given $E\rightarrow X$ is a twisted weighted vector bundle with not necessarily distinct weights $\mathbf{w}:=(w_1,\dots,w_d)$, the total weight is $w:=\sum_{i=1}^d w_i$.

\begin{Lemma}\label{lem:cohomology-proj-bundle}
    Let $g:\cP \rightarrow \cX$ be a twisted weighted projective bundle with total weight $w$. Given two integers $-w < n \leq 0$ and $m\geq 0$, then
    $$ R^mg_*\sO_{\cP}(n) = \begin{cases}
        \sO_{\cX} \qquad {\rm if }\, n=m=0 \\
        0 \qquad \, {\rm otherwise}
    \end{cases}$$
    Moreover, we have that $R^mg_*\sO_{\cP}(n)=0$ for every $n\geq 0$ and $m\geq 1$.
\end{Lemma}
\begin{proof}
  If $m\neq 0$ or $n\neq0$, the question is clearly local. If $n=m=0$, the question is also local since we have the canonical morphism $\sO_{\cX}\rightarrow g_*\sO_{\cP}$. Thus, we can assume $\cX$ to be affine. Then the isomorphism follows from the computations done in \cite[\S1.4]{Dol}.
\end{proof}

\begin{Cor}\label{cor:push-forward-of-1}
    In the setting of \S\ref{sec: notation-resolution}, the natural morphism $\sO_{\cY}\rightarrow Rf_*\sO_{\widetilde{\cY}}$ is a quasi-isomorphism of complexes. As a consequence, $f_*(1)=1$ in $\Kgroup(\cY)$.
\end{Cor}

\begin{proof}
    Note that, since everything is normal, we have that $\sO_{\cY} \rightarrow f_*\sO_{\widetilde{\cY}}$ is an isomorphism of sheaves, therefore it is enough to show that $R^if_*\sO_{\widetilde{\cY}}=0$ for $i>0$. 
    Since weighted blowups commute with flat base change (see \cite[Corollary 3.2.4]{quek-rydh-weighted-blowup}) and being the statement local on $\cY$, we can assume $\cY$ is the spectrum of an affine local ring $(A,\mathfrak{m})$. Let us denote by $I\subset A$ the (complete intersection) ideal associated to $\cX \subset \cY$. If $I \not\subset \mathfrak{m}$ (i.e. $I=A$), then the morphism $f$ is actually an isomorphism by construction and there is nothing to prove. Thus we can assume $I\subset \mathfrak{m}$. Therefore the $I$-completion morphism $\spec \widehat{A} \rightarrow \spec A$ is faithfully flat, and we can reduce to the case when $A$ is $I$-complete. We can then use the Theorem on Formal Functions for proper algebraic stacks (see for instance \cite[Theorem 1.4]{Olss}) to get the isomorphism of $A$-modules
    $$ R^if_*\sO_{\widetilde{\cY}}\simeq  \lim_n\oH^i\left(\widetilde{\cY}, \sO_{\widetilde{\cY}}/I^{n+1}\right)$$
    and thus, it is enough to prove that the right-hand side vanishes. In contrast with the classical case, the fiber over $\spec(A/I) \subset \spec A$ is not the exceptional divisor $\widetilde{\cX}$, but an infinitesimal thickening of it. We claim that this is enough to show that
    $$\lim_n\oH^i\left(\widetilde{\cY}, \sO_{\widetilde{\cY}}/I^{n+1}\right) \simeq \lim_n\oH^i\left(\widetilde{\cY}, \sO_{\widetilde{\cY}}/I_{\widetilde{\cX}}^{n+1}\right).$$
    Indeed, by the construction of the weighted blowup, there exists an integer $N>0$ such that $I^N\sO_{\widetilde{\cY}}\subset I_{\widetilde{\cX}}^N\subset I\sO_{\widetilde{\cY}} $  thus the claim follows by the functoriality of the $i$-th cohomology group.
    Finally, it remains to show that $ \lim_n\oH^i(\widetilde{\cY}, \sO_{\widetilde{\cY}}/I_{\widetilde{\cX}}^{n+1})=0$ for $i>0$. To do that, we will show that $\oH^i(\widetilde{\cY}, \sO_{\widetilde{\cY}}/I_{\widetilde{\cX}}^{n+1})=0$ for every $n$ and we will proceed by induction on $n$. 
    If $n=0$, the result follows from Lemma~\ref{lem:cohomology-proj-bundle}. Suppose now that $n>0$. Then, we have the exact sequence of $\sO_{\widetilde{\cY}}$-modules
    $$ 0 \longrightarrow I_{\widetilde{\cX}}^n/I_{\widetilde{\cX}}^{n+1} \longrightarrow \sO_{\widetilde{\cY}}/I_{\widetilde{\cX}}^{n+1} \longrightarrow \sO_{\widetilde{\cY}}/I_{\widetilde{\cX}}^{n}\longrightarrow 0$$
    and, taking cohomology, we get the exact sequence
    $$ \oH^i\left( \widetilde{\cX}, I_{\widetilde{\cX}}^n/I_{\widetilde{\cX}}^{n+1} \right) \longrightarrow  \oH^i\left( \widetilde{\cY},\sO_{\widetilde{\cY}}/I_{\widetilde{\cX}}^{n+1}  \right)  \longrightarrow \oH^i\left( \widetilde{\cY},\sO_{\widetilde{\cY}}/I_{\widetilde{\cX}}^{n}  \right).$$
    The right-hand term is $0$ by the inductive argument. Moreover, since $\widetilde{\cX}\subset \widetilde{\cY}$ is the exceptional divisor, we have that $I_{\widetilde{\cX}}^n/I_{\widetilde{\cX}}^{n+1}\simeq (I_{\widetilde{\cX}}/I_{\widetilde{\cX}}^2)^n\simeq \sO_{\widetilde{\cX}}(n)$. Finally, the left-hand term is $0$ by Lemma~\ref{lem:cohomology-proj-bundle}.
    \end{proof}
\begin{Cor}\label{cor: push-forward-upper-shriek}
    In the setting of \S\ref{sub: Gysin for Weighted Blowups}, we have:
        $$g_*f^!(\alpha)=\alpha$$ for each $\alpha \in \Kring(\cX)$.
\end{Cor}
\begin{proof}
    Thanks to Proposition~\ref{prop:existence of gamma}, it is enough to show that $g_*f^{!}(1)=1$. Thanks to Proposition~\ref{prop: f upper shriek}, we have that
    $$ f^{!}(1)= \frac{p(t^{-1})-p(1)}{t-1}$$
    with $t=[\sO_{\tcX}(1)]$. Recall that $p(t^{-1})= \sum_{i=0}^{w} b_it^{-i}$ is a polynomial in $t^{-1}$ of degree $w$ where $w$ is the total weight and every $b_i$ is in the image of the pullback of $g$. Therefore, we have
      \begin{align*}
         f^{!}(1) & = \frac{  \sum_{i=0}^{w} b_it^{-i} -{ \sum_{i=0}^{w} b_i}}{t-1} \\
         & =(-t^{-1})  \sum_{i=1}^{w} b_i\frac{(t^{-i}-1)}{t^{-1}-1} \\
         & = (-t^{-1})\sum_{i=1}^{w} b_i \sum_{j=0}^{i-1} t^{-j} \\
         & = -\sum_{j=0}^{w-1} t^{-j-1} \left(\sum_{i=j+1}^{w} b_i\right).
     \end{align*}
     Since $p(t^{-1})$ is zero in $\Kring(\widetilde{\cX})$, we can add it to the previous expression to get
     \begin{align*}
         f^{!}(1) &= -\sum_{j=1}^{w} t^{-j} \left(\sum_{i=j}^{w} b_i\right) + \sum_{j=0}^{w} b_jt^{-j}= b_0 - \sum_{{j=1}}^{w-1}t^{-j} \left(\sum_{i=j+1}^{w} b_i\right).
     \end{align*}
     The result follows applying $g_*$ and noticing that $b_0=1$, $g_*(1)=1$ and $g_*t^{-j}=0$ for every $j=1,\dots,w-1$, thanks to Lemma~\ref{lem:cohomology-proj-bundle}.
\end{proof}
\begin{Prop}\label{prop:fulton}
    In the setting of \S\ref{sub: Gysin for Weighted Blowups}, we have the following results:
    \begin{enumerate}[label=$(\alph*)$]
        \item\label{point: a} for every $\alpha \in \Kgroup(\cX)$, we have that $f^*i_*(\alpha)=j_*f^!(\alpha)$ in $\Kgroup(\widetilde{\cY})$;
        \item\label{point: b} for every $\beta \in  \Kgroup(\cY)$, we have that $f_*f^*(\beta)=\beta$ in $\Kgroup(\cY)$;
        \item\label{point: c} if $\widetilde{\alpha} \in \Kgroup(\widetilde{\cX})$ such that $g_*(\widetilde{\alpha})=0$ and $j_*(\widetilde{\alpha})=0$, then $\widetilde{\alpha}=0$.
    \end{enumerate}
    \end{Prop}
\begin{proof}
(a) This follows from Proposition~\ref{prop: properties pullbacks}\eqref{point 1: prop: properties pullbacks}, as the blowup factors as the composition of a regular closed embedding and a flat morphism by Lemma~\ref{lem:lci-factorization}.

(b) The equality is an immediate consequence of the projection formula in equation~\eqref{eqn: Projection Formula} and of Corollary~\ref{cor:push-forward-of-1}.

(c) Recall that $p(t^{-1})=\sum_{i=0}^w b_i t^{-i}$ is zero in $\Kring(\widetilde{\cX})$. In particular, every element in $\Kring(\widetilde{\cX})$ can be represented as a polynomial in $t^{-1}$. Moreover, note that $b_w = \prod_i \det([E_i^\vee])$ is invertible. Here $E$ is the weighted conormal cone of $\cX \subseteq \cY$, and the $E_i$ are the vector bundles associated to it in Remark~\ref{rmk: mistake in Arena-Obinna}. In particular, we can also bound the degree in $t^{-1}$ of the polynomials by $w-1$, and write $\widetilde \alpha = \sum_{i=0}^{w-1} a_it^{-i}$. Such representation is indeed unique.
By applying $g_*$ we obtain $$g_*\widetilde \alpha=a_0=0.$$ By Theorem~\ref{prop: self intersection formula}, and because $a_0=0$,
$$
0=j^*j_*\widetilde \alpha = (1-t)\widetilde \alpha = \sum_{i=0}^{w-2}(a_i-a_{i+1})t^{-i}+a_{w-1}t^{-(w-1)}.$$
By uniqueness of this expression, we have $a_i=0$ for all $i$, thus $\widetilde\alpha=0$.
\end{proof}

One upside of working in the K-theoretic setting is that the pushforward of coherent sheaves along proper morphisms is defined also in the non-representable case, in contrast of what happens for the proper pushforward of integral Chow-theoretic cycles, see for instance \cite[Example 6.2]{Arena-Obinna}. This ensures that the blowup exact sequence is split exact, as the following corollary shows. 

\begin{Cor}\label{cor: main exact sequence}
    In the setting of \S\ref{sec: notation-resolution}, there exists a split short exact sequence
    $$ 
  \begin{tikzcd}
0 \arrow[r] & \Kgroup(\cX) \arrow[r, "{(f^!, i_*)}"] & \Kgroup(\widetilde{\cX})\oplus \Kgroup(\cY)  \arrow[r, "-j_*+f^*"] & \Kgroup(\widetilde{\cY}) \arrow[r] & 0
\end{tikzcd}
    $$
    where a left inverse of the first map is given by $(\widetilde{\alpha},\beta)\mapsto g_*(\widetilde{\alpha})$.
\end{Cor}

\begin{proof}
    We start by proving the surjectivity of the right-hand side morphism. Let $\widetilde{\beta} \in \Kgroup(\widetilde{\cY})$, and consider the cartesian diagram
    $$
    \begin{tikzcd}
    \widetilde{\cY}\setminus \widetilde{\cX} \arrow[r, "v", hook] \arrow[d, "f^{\circ}"] & \widetilde{\cY} \arrow[d, "f"] \\
    \cY \setminus \cX \arrow[r, "u", hook]                                       & \cY                       
    \end{tikzcd}
    $$
    where $f^{\circ}$ is an isomorphism. By Proposition~\ref{prop: properties pullbacks}, we have that $u^{*}f_*=f^{\circ}_*v^*$. Thus, if we denote by $\widetilde{\beta}'$ the element $f^*f_*\widetilde{\beta}-\widetilde{\beta}$, we have that 
    $$ v^*\widetilde{\beta}'= v^*f^*f_*\widetilde{\beta}-v^*\widetilde{\beta}= f^{\circ,*}u^*f_*\widetilde{\beta}-v^*\widetilde{\beta}=f^{\circ,*}f^{\circ}_*v^*\widetilde{\beta}-v^*\widetilde{\beta}=0$$
    since $f^{\circ}$ is an isomorphism. By the localization sequence in $K$-theory, we get that $\widetilde{\beta}'=j_*(\widetilde{\alpha})$ for some $\alpha \in \Kgroup(\widetilde{\cX})$ and thus we have that $\widetilde{\beta}=f^{*}f_*\widetilde{\beta}-j_*{\widetilde{\alpha}}$ and we get the surjectivity.

    We will now show the exactness in the middle. Given $\alpha\in \Kgroup(\cX)$, we have that $f^*i_*(\alpha)=j_*f^!(\alpha)$ thanks to Proposition~\ref{prop:fulton}\ref{point: a}, thus $\im(f^!,i_*)\subset \ker(-j_*+f^*)$. Moreover, suppose we are given $(\widetilde{\alpha},\beta) \in \Kgroup(\widetilde{\cX})\oplus \Kgroup(\cY)$ such that $j_*\widetilde{\alpha}=f^*\beta$. Part \ref{point: b} of Proposition~\ref{prop:fulton} tells us that $\beta=f_*j_*(\widetilde{\alpha})=i_*g_*(\widetilde\alpha)$, therefore it remains to show that $f^!g_*(\widetilde\alpha)=\widetilde\alpha$. Note that
    \begin{itemize}
        \item $g_*\left(f^!g_*(\widetilde\alpha)-\widetilde\alpha\right)=0$ because of Corollary~\ref{cor: push-forward-upper-shriek};
        \item $j_*\left(f^!g_*(\widetilde\alpha)-\widetilde\alpha\right)=f^*i_*g_*(\widetilde\alpha)-j_*\widetilde\alpha=f^*\beta-j_*\widetilde\alpha=0$ because of Proposition~\ref{prop:fulton}\ref{point: a}.
    \end{itemize}
    Hence, Proposition~\ref{prop:fulton}\ref{point: c} gives us that $f^!g_*(\widetilde\alpha)=\widetilde\alpha$ and thus $\im(f^!,i_*)= \ker(-j_*+f^*)$.

    Finally, the remaining statements follow from Corollary~\ref{cor: push-forward-upper-shriek}.
\end{proof}

\subsection{Proof of Theorem~\ref{thm: main theorem}}\label{sec:proof-main-theorem}

The short exact sequence in Corollary~\ref{cor: main exact sequence} provides a presentation of the group $\Kgroup(\widetilde \cY)$. We are only left to describe how its elements multiply. 
For example, for $q(t,t^{-1}), q'(t,t^{-1})\in \Kring(\widetilde \cX)$, by the projection and self intersection formulas, we have
\begin{align*}
    (q(t,t^{-1}),0) \cdot (q'(t,t^{-1}),0) & = j_*(q(t,t^{-1}))\cdot j_*(q'(t,t^{-1}))\\ 
    & = j_*(q(t,t^{-1})\cdot j^*j_*(q'(t,t^{-1}))\\
    & = j_*(q(t,t^{-1}) \cdot q'(t,t^{-1})\cdot (1-t))\\
    & = (q(t,t^{-1})\cdot q'(t,t^{-1})\cdot (1-t),0).
\end{align*}
The others are similar.\qed

\subsection{Keel's Formula}\label{sub:Keel-formula}

When the restriction morphism $i^*:\Kring(\cY)\rightarrow\Kring(\cX)$ is surjective, the presentation of $\Kring(\widetilde{\cY})$ in Theorem~\ref{thm: main theorem} simplifies notably, as we will prove below. For classical blowups, the Chow-theoretic version is a theorem of Keel (\cite[Theorem 1, Appendix, page 571]{Keel}), later generalized in~\cite[Corollary 6.5]{Arena-Obinna} to weighted blowups.

Before stating and proving the K-theoretic weighted Keel formula, let us set up some notation.

Recall that for $\cX$ and $\cY$ as in~\S\ref{sec: notation-resolution}, the restriction of $\sO_{\tcY}(1)\simeq\sO_{\tcY}(-\tcX)$ to $\tcX$ is $\sO_{\tcX}(1)$ on $\tcX$. Now, suppose the restriction $i^*:\Kring(\cY)\rightarrow\Kring(\cX)$ to be surjective. This map can be naturally extended to a surjective morphism of rings $i_s^*:\Kring(\cY)[s^{\pm1}]\rightarrow\Kring(\cX)[t^{\pm1}]$, sending $s$ to $t$. Recall that $p(t^{-1})=\lambda_{-1}^{\Gm}(\sN_{\cX/\cY}^{\vee})$, and $p(1)=\lambda_{-1}(\sN_{\cX/\cY}^{\vee}) \in \Kring(\cX)$.

\begin{Lemma}\label{lem: uniqueness r(s)}
    There is a well-defined homomorphism of groups
    \[
        \tau:\Kring(\cX)[t^{\pm1}]\longrightarrow\frac{\Kring(\cY)[s^{\pm1}]}{((1-s)\cdot\ker i^*)}
    \]
    defined by sending $\alpha\in\Kring(\cX)[t^{\pm1}]$ to $\beta\cdot(1-s)$ for any choice of $\beta\in\Kring(\cY)$ such that $i^*\beta=\alpha$.
    In particular,
    \[
    r(s)=\tau\left(\frac{p(t^{-1})-p(1)}{1-t}\right),
    \]
    is the unique element of $\Kring(\cY)[s^{\pm1}]/((1-s)\cdot\ker i^*)$ divisible by $1-s$ and such that $i_s^*r(s)=p(t^{-1})-p(1)$.
\end{Lemma}

\begin{proof}
    The first part is obvious. The second part follows from the fact that if $a \in \Kring(\cY)[s^{\pm1}]$ is such that $i_s^*((1-s) a)=(1-t)b$, then $i_s^*(a)=b$.
\end{proof}

    We set
    \begin{equation}\label{eq: Q(s)}
        Q(s)=r(s)+i_*(1)\in\frac{\Kring(\cY)[s^{\pm1}]}{((1-s)\cdot\ker i^*)}.
    \end{equation}
    Note that $i_s^*$ induces an isomorphism
    \begin{equation}\label{eq: isom Xtilde Keel formula}
        \Kring(\tcX)\simeq\frac{\Kring(\cX)[t^{\pm1}]}{(p(t^{-1}))}\xleftarrow[\psi]{\simeq}\frac{\Kring(\cY)[s^{\pm1}]}{(\ker i^*,Q(s))}
    \end{equation}
    as the image of $i_s^*(Q(s))$ in $\Kring(\tcX)\simeq\Kring(\cX)[t^{\pm1}]/(p(t^{-1}))$ is $p(t^{-1})-p(1)+i^*i_*(1)=p(t^{-1})$, by the self-intersection formula.
    
\begin{Corollary}[Keel Formula]\label{cor: keel formula}
    Let $\cX$ and $\cY$ as in \S\ref{sec: notation-resolution}, and suppose the restriction $i^*:\Kring(\cY)\rightarrow\Kring(\cX)$ to be surjective. With the notation above, the morphism
    \[
    \begin{tikzcd}
        \Phi:\Kring(\cY)[s^{\pm1}]\arrow[r] & \Kring(\tcY), & s\mapsto {[\sO_{\tcY}(1)]},
    \end{tikzcd}
    \]
    induced by the pullback along $f$ is surjective, and yields an isomorphism
    \[
        \Kring(\tcY)\simeq\frac{\Kring(\cY)[s^{\pm1}]}{((1-s)\cdot\ker i^*,Q(s))}.
    \]
\end{Corollary}
\begin{proof}
    Throughout the whole proof, we will identify $s$ with its image $\Phi(s)=[\sO_{\tcY}(1)]$.
    First, we prove that the morphism $\Phi$ is surjective, using Corollary~\ref{cor: main exact sequence}. Let $\widetilde{x}=\sum_i g^*(x_i)t^i\in\Kring(\tcX)$, and $y,y_i\in\Kring(\cY)$ with $i^*y_i=x_i$. Then,
    \begin{align*}
        -j_*(\widetilde{x})+f^*y &= -\sum_i j_*(g^*i^*(y_i)\cdot j^*(s^i))+f^*y\\
        &= -\sum_i f^*(y_i)s^ij_*(1)+f^*y\\
        &= -\sum_i f^*(y_i)s^i(1-s)+f^*y\in\operatorname{Im}(\Phi),
    \end{align*}
    where we have used that $j_*(1)=1-s$. Moreover, note that for every $y\in\ker i^*$
    \begin{align*}
        \Phi((1-s)y)=j_*(1)f^*(y)=j_*(g^*i^*(y))=0.
    \end{align*}
    It follows that $(1-s)\cdot\ker i^*\subset\ker\Phi$, hence we get a morphism
    \[
        R_1:=\frac{\Kring(\cY)[s^{\pm1}]}{((1-s)\cdot\ker i^*)}\longrightarrow\Kring(\tcY),
    \]
    that we still denote by $\Phi$. Let $Q(s)\in R_1$ be the polynomial defined in~\eqref{eq: Q(s)}, and let $q(s)\in R_1$ be any polynomial such that $i_s^*q(s)=\frac{p(t^{-1})-p(1)}{t-1}$. By uniqueness of $r(s)$ in Lemma~\ref{lem: uniqueness r(s)}, we have $Q(s)=q(s)(s-1)+i_*(1)$, thus
    \begin{align*}
        \Phi(Q(s))&=-(f^*q)(s)\cdot(1-s)+f^*i_*(1)\\
        &=-(f^*q)(s)\cdot j_*(1)+f^*i_*(1)\\
        &=-j_*j^*((f^*q)(s))+f^*i_*(1)\\
        &=-j_*f^!(1)+f^*i_*(1)=0.
    \end{align*}
    Here, we have abused notation and suggestively wrote $f^*(q)(s)$ for $\Phi(q(s))$.
    We have used Proposition~\ref{prop: f upper shriek} to compute $f^!(1)$, and the last expression vanishes by Proposition~\ref{prop:fulton}. It follows that $\Phi$ factors through
    \[
        R:=\frac{\Kring(\cY)[s^{\pm1}]}{((1-s)\cdot\ker i^*,Q(s))}\longrightarrow\Kring(\tcY),
    \]
    that we again denote by $\Phi$. To show this to be an isomorphism, we define a surjective group homomorphism $\varphi:\Kring(\tcX)\oplus\Kring(\cY)\rightarrow R$ fitting in a commutative diagram of abelian groups
    \[
    \begin{tikzcd}
        \Kring(\tcX)\oplus\Kring(\cY)\arrow[d,twoheadrightarrow,"\varphi"]\arrow[r,"-j_*+f^*"] & \Kring(\tcY)\\
        R\arrow[r] & R/\!\ker\Phi\arrow[u,"\simeq"]
    \end{tikzcd}
    \]
    analogous to that in~\cite[page 573]{Keel}.
    On $\Kring(\cY)$, we define $\varphi$ to be the composition of $\Kring(\cY)\hookrightarrow\Kring(\cY)[s^{\pm1}]\twoheadrightarrow R$. In terms of the isomorphism in Equation~\eqref{eq: isom Xtilde Keel formula}, on the summand $\Kring(\tcX)$, we define $\varphi$ as the multiplication by $-(1-s)$.
    The commutativity of the diagram and the surjectivity of $\varphi$ follow from a formal computation. To conclude the proof, it is then enough to show that the composition
    \[
    \begin{tikzcd}[column sep=large]
        \Kring(\cX)\arrow[r,"{(f^{!},i_*)}"] & \Kring(\tcX)\oplus\Kring(\cY)\arrow[r,"\varphi"] & R
    \end{tikzcd}
    \]
    is the zero-map. Let $x\in\Kring(\cX)$ and $y\in\Kring(\cY)$ such that $i^*y=x$. Then,
    \begin{align*}
        \varphi\circ(f^!,i_*)(x)&=\varphi(f^!(i^*y),i_*i^*y)\\
        &=\varphi(g^*i^*(y)f^!(1),y\cdot i_*(1))\\
        &=y\cdot(\varphi(f^!(1))+i_*(1))\\
        &=y\cdot Q(s)=0.
    \end{align*}
    In the last equality, we have used that $\varphi(f^!(1))$ is the image of $r(s)$ under the quotient map $R_1\rightarrow R$.
\end{proof}

\section{Applications}\label{sec:applications}

\subsection{K-theory of $\overline{\cM}_{1,2}$}\label{sub:K-theory-M-12}

As an application, we compute the K-theory of the stack $\overline{\cM}_{1,2}$ parametrizing stable genus 1 curves with 2 marked points. We will leverage the following beautiful description of $\overline{\cM}_{1,2}$ as a weighted blowup due to Inchiostro (\cite{Inchiostro}).

Consider the weighted projective stack $\cP(2,3,4)\simeq[(\A^3\setminus0)/\Gm]$ where the $\Gm$-action has weights $(2,3,4)$ on the coordinates $x_0,x_1,x_2$. Let $Z\subset\cP(2,3,4)$ be the closed subscheme defined by the equations $x_0^3-x_1^2=x_2=0$. This is a closed point with trivial stabilizer, in particular the normal bundle of $Z$ in $\cP(2,3,4)$ is trivial. See~\cite{Inchiostro} for an elegant modular interpretation of the following result.

\begin{Thm}[{\cite[Theorem 2.10]{Inchiostro}}]
    Suppose the base field $k$ is of characteristic not dividing 6. There is an isomorphism $\overline{\cM}_{1,2}\simeq Bl_{Z}^{(4,6)}(\cP(2,3,4))$, where $Bl_{Z}^{(4,6)}(\cP(2,3,4))$ is the weighted blowup of $\cP(2,3,4)$ along $Z$, with weights $(4,6)$ and exceptional divisor $\cE\simeq\cP(4,6)$.
\end{Thm}

We are ready to compute its K-theory. For the Chow-theoretic version of this computation, see~\cite[Theorem 4.12]{Inchiostro} or~\cite[Proposition 6.7]{Arena-Obinna}.

\begin{proof}[Proof of Theorem~\ref{thm: K-theory M12bar}]
    First, by Theorem~\ref{prop: twisted weighted projection formula}, we have
    \[
        \Kring(\cP(2,3,4))\simeq\frac{\Z[u^{\pm1}]}{((1-u^{-2})(1-u^{-3})(1-u^{-4}))}
    \]
    where $u=[\sO(1)]$. Moreover, the restriction morphism $i^*:\Kring(\cP(2,3,4))\rightarrow\Kring(Z)\simeq\Z$ is the rank map, hence $i^*$ is surjective and has kernel generated by $1-u$. Therefore, by Corollary~\ref{cor: keel formula}, we have
    \[
        \Kring(\overline{\cM}_{1,2})\simeq\frac{\Kring(\cP(2,3,4))[s^{\pm1}]}{((1-s)\cdot\ker i^*,Q(s))}\simeq\frac{\Z[u^{\pm1},s^{\pm1}]}{((1-u^{-2})(1-u^{-3})(1-u^{-4}),(1-s)(1-u),Q(s))}.
    \]
    Recall that $Q(s)=r(s)+i_*(1)$, where $r(s)$ is a lift of $p(t^{-1})-p(1)\in\Kring(Z)[t^{\pm1}]=\Z[t^{\pm1}]$ along the restriction map $\Kring(\cP(2,3,4))[s^{\pm1}]\rightarrow\Kring(Z)[t^{\pm1}]$, $s\mapsto t$, and $p(t^{-1})=\lambda_{-1}^{\Gm}(\sN_{Z/\cP(2,3,4)}^{\vee})$. Now, $p(t^{-1})=(1-t^{-4})(1-t^{-6})$, and
    \[
        i_*(1)=[\sO_Z]=[\sO_{\{x_0^3-x_1^2=0\}}]\cdot[\sO_{\{x_2=0\}}]=(1-u^{-6})(1-u^{-4}).
    \]
    Therefore, we can take $Q(s)=(1-s^{-4})(1-s^{-6})+(1-u^{-4})(1-u^{-6})$, as $p(1)=0$.

    It remains to describe the generators $s$ and $u$ in geometric terms. By construction, the exceptional divisor corresponds to the boundary divisor $
    \Delta_1$, parametrizing curves with an elliptic tail (see \cite[Theorem 2.10]{Inchiostro}), which also coincides with the universal section of the universal curve morphism $\overline{\cM}_{1,2} \rightarrow \overline{\cM}_{1,1}$. Thus we have $s=[\sO(-\Delta_1)]$. 

    We claim instead that $u = \mathbb{E}(\Delta_1)$ is the Hodge bundle twisted by $\Delta_1$. To see this, we proceed as follows. Let $f \colon \overline{\mathcal{M}}_{1,2} \simeq \operatorname{Bl}_{Z}^{(4,6)}(\mathcal{P}(2,3,4)) \to \mathcal{P}(2,3,4)$ be the blowup map. Outside of $Z$, we have
\[
f^* \mathcal{O}_{\cP(2,3,4)}(1) \simeq \mathbb{E}.
\]
Indeed, the restriction of the forgetful map $\overline{\mathcal{M}}_{1,2} \to \overline{\mathcal{M}}_{1,1}$ to $\overline{\mathcal{M}}_{1,2} \setminus \Delta_1 \simeq \mathcal{P}(2,3,4) \setminus Z$ is a map
\[
\mathcal{P}(2,3,4) \setminus Z \to \mathcal{P}(4,6)
\]
obtained from an explicit $\mathbb{G}_m$-equivariant rational map $\mathbb{A}^2 \setminus \{0\} \to \mathbb{A}^3 \setminus \{0\}$ described in \cite[\S2.1]{Inchiostro}. In particular, the pullback of $\mathcal{O}_{\mathcal{P}(4,6)}(1)$ to $\mathcal{P}(2,3,4) \setminus Z$ is $\mathcal{O}_{\mathcal{P}(2,3,4)}(1)$. By Example~\ref{example M11}, the line bundle $\mathcal{O}_{\mathcal{P}(4,6)}(1)$ corresponds to the Hodge bundle on $\overline{\mathcal{M}}_{1,1}$ under the isomorphism $\mathcal{P}(4,6) \simeq \overline{\mathcal{M}}_{1,1}$, so its pullback to $\overline{\mathcal{M}}_{1,2} \setminus \Delta_1$ is also the Hodge bundle.

It follows that
\[
f^* \mathcal{O}_{\cP(2,3,4)}(1) \simeq \mathbb{E}(m \Delta_1)
\]
for some $m \in \mathbb{Z}$.

Identifying $\Delta_1$ with $\overline{\mathcal{M}}_{1,1}$ via the section $\overline{\mathcal{M}}_{1,1} \to \overline{\mathcal{M}}_{1,2}$, the restriction satisfies
\[
\mathbb{E}|_{\Delta_1} \simeq \mathcal{O}_{\mathcal{P}(4,6)}(1).
\]
Moreover,
\[
\mathcal{O}(\Delta_1)|_{\Delta_1} \simeq \mathcal{N}_{\Delta_1 / \overline{\mathcal{M}}_{1,2}} \simeq \psi_1^\vee \simeq \mathbb{E}^\vee \simeq \mathcal{O}_{\mathcal{P}(4,6)}(-1),
\]
where $\psi_1$ denotes the psi class on $\overline{\mathcal{M}}_{1,1}$. The isomorphism $\mathbb{E} \simeq \psi_1$ is given by the evaluation map at the first marking.

In conclusion, since $\Delta_1$ is contracted by $f$, we have
\[
\mathcal{O} \simeq f^* \mathcal{O}_{\cP(2,3,4)}(1)|_{\Delta_1} \simeq \mathbb{E}(m \Delta_1)|_{\Delta_1} \simeq \mathcal{O}_{\cP(2,4)}(1 - m),
\]
from which we deduce $m = 1$.
\end{proof}

\begin{Rmk}
    Contrarily to what happens for the Chow ring, the relation $(1-u^{-2})(1-u^{-3})(1-u^{-4})$ defining the K-theory of $\cP(2,3,4)$ is not redundant. A simple way to see this is to quotient $\Kring(\overline{\cM}_{1,2})$ by $1-s$. 
\end{Rmk}

\subsection{The Lambda Polynomial of the Tangent complex}\label{sub:lambda-poly}

Let $f:\widetilde{\cY} \to \cY$ be a regular weighted blowup centered at $\cX \subseteq \cY$ with non-necessarily distinct weights $\mathbf{w} \in \mathbb{N}^d$. In this section, we determine the Lambda polynomial $\lambda_z(\bT_{\widetilde{\cY}})$ of the tangent complex of $\widetilde{\cY}$ in terms of the blowup data.

\begin{Rmk}\label{rmk: tangent complex stacks}
    Recall that every Artin stack $\cZ$ has a well-defined associated cotangent complex $\bL_{\cZ}^\bullet \in D^b(\cZ)$ in the derived category of bounded complexes on $\cZ$. The tangent complex is by definition $\bT_{\cZ}:=(\bL_{\cZ}^\bullet)^\vee $. If $\cZ=[Z/G]$ is the quotient of a scheme $Z$ by an affine smooth group $G$ and $\pi: Z \to \cZ$ is the quotient map, then one has a short exact sequence 
    $$
        \pi^*\mathbb{L}_{\mathcal{Z}}^\bullet \to \mathbb{L}_Z^\bullet \to \mathbb{L}_{Z/\mathcal{Z}}^\bullet \xrightarrow{+1}
    $$
    where $\mathbb{L}_X^\bullet$ is the standard cotangent complex of the scheme $X$ and $\mathbb{L}_{X/\mathcal{Z}}^\bullet$ is the relative cotangent complex of the map $\pi$.  Because $Z \to [Z/G]$ is a $G$-torsor, the map is smooth and the relative cotangent complex is just the vector bundle generated by the differentials along the fibers of the group action. If we let $\mathfrak{g}$ denote the Lie algebra of $G$, the infinitesimal action of $G$ on $Z$ provides a natural homomorphism of vector bundles
    $$
        a: \mathfrak{g} \otimes \sO_Z \to T_Z
    $$
    whose image is $T_{Z/\cZ}$. Dually, we obtain isomorphisms
    $$
        \mathbb{L}_{Z/\mathcal{Z}}^\bullet \simeq \Omega^1_{Z/\mathcal{Z}} \simeq \mathfrak{g}^\vee \otimes \sO_Z.
    $$
    If furthermore $Z$ is smooth, $\mathbb{L}_Z^\bullet$ is quasi-isomorphic to the Kähler differentials $\Omega_Z^1$ situated in degree $0$. In this case, $\pi^*\mathbb{L}_{\mathcal{Z}}^\bullet$ has a simple representation as a $G$-equivariant two-term complex in degrees $0$ and $1$:
    $$
        \pi^*\mathbb{L}_{\mathcal{Z}}^\bullet \simeq [ \Omega_Z^1 \xrightarrow{d} \mathfrak{g}^\vee \otimes \sO_Z ].
    $$
\end{Rmk}

Recall also that the Lambda polynomial of a perfect complex on a smooth global quotient Artin stack $\cZ$ is defined as follows.

\begin{Def}
    Let $\sE$ be a perfect complex on $\cZ$. The Lambda polynomial of $\sE$ is
    $$
        \lambda_z(\sE) = \sum_{i = 0}^\infty [\lambda^i \sE] z^i \in K^0(\mathcal{Z})[\![ z ]\!].
    $$    
\end{Def}

Note that this is always invertible. By setting $z = -1$, one obtains the K-theoretic Euler class as defined in Definition~\ref{def:euler_class}.

Recall that we set 
$$Q_z(t) := \lambda_z(\sN_{\cX/\cY}) \in \Kring_{\mathbb{G}_m}(\cX)[\![ z ]\!]= \Kring(\cX \times B\Gm)[\![ z ]\!].$$

and defined
$$
q: \Kring(\cX)[t^{\pm 1} ] [\![ z ]\!] \oplus \Kring(\cY)[\![ z ]\!] \longrightarrow \Kring(\widetilde{\cY})[\![ z ]\!]
$$
as the quotient map induced by the presentation of $\Kring(\widetilde{\cY})$ in Theorem \ref{thm: main theorem}.

We are ready to prove Theorem~\ref{thm: Lambda_poly_tangent}.

\begin{proof}[Proof of Theorem~\ref{thm: Lambda_poly_tangent}]
    
First note that because the restriction of $(\lambda_z(\bT_{\widetilde{\cY}})/\lambda_z(\bT_{\cY}))-1$ to $\Kring(\widetilde{\cY} \smallsetminus \widetilde{\cX})[\![ z ]\!]$
vanishes, we can write 
\begin{equation}\label{eqn: def_alpha}
\frac{\lambda_z(\bT_{\widetilde{\cY}})}{\lambda_z(f^* \bT_{\cY})} =1+ z j_*(\alpha)
\end{equation}
for some $\alpha \in \Kring(\widetilde{\cX})[\![ z ]\!]$ to be determined. Moreover, the element $\alpha$ is functorial for smooth morphisms $\cY' \to \cY$ and closed embeddings $\cY' \to \cY$ that meet $\cX$ transversely, in the following sense. 

Given a Tor-independent cartesian diagram 
\begin{center}
\begin{tikzcd}
\widetilde{\cY}' \arrow[r, "\widetilde{h}"] \arrow[d, "f'"] & \widetilde{\cY} \arrow[d, "f"] \\
\cY' \arrow[r, "h"] & \cY
\end{tikzcd}
\end{center}
one has $h^* \bT_{f}=\bT_{f'}$ and thus
\begin{equation*}
1+z (j')_*(h^*(\alpha))=\lambda_z(h^* \bT_{f})= \lambda_z(\bT_{f'})
\end{equation*}

Here, by abuse of notation, $h$ also denotes the base change $\widetilde{\cX}'=\widetilde{\cX} \times_{\cY} \cY' \to \widetilde{\cX}$ and $j'$ denotes the inclusion $\tcX' \to \tcY'$. All pullbacks and pushforwards are derived. This is saying that if $\alpha$ satisfies equation~\eqref{eqn: def_alpha}, then $h^*(\alpha)$ satisfies the same Equation for $\widetilde{\cY}' \to \cY'$.
From Remark~\ref{rem:blowup-of-def-cone}, we have a commutative diagram
\[
\begin{tikzcd}[sep=.6cm]
& \widetilde{\cX}\times\mathbb{A}^1
    \arrow[rr]
    \arrow[ddrr, "\textstyle \cZ" description]
&& \cZ\times\mathbb{A}^1
    \arrow[rr]
    \arrow[dd]
&& \widetilde{\cD}
    \arrow[dd, "b_D"]
\\
\widetilde{\cX}
    \arrow[rr]
    \arrow[ddrr]
    \arrow[ur]
&& \cZ
    \arrow[rr]
    \arrow[dd]
    \arrow[ru]
&& \widetilde{\cD}_t
    \arrow[ru]
\\
&&&
\cX \times\mathbb{A}^1
    \arrow[rr]
&& D_{\cX_\bullet}\cY
\\
&&
\cX
    \arrow[rr]
    \arrow[ru]
&& (D_{\cX_\bullet}\cY)_t
    \arrow[ru]
    \arrow[from=uu, crossing over]
\end{tikzcd}
\]
where
$
\widetilde{D}_t \longrightarrow (D_{\cX_\bullet}\cY)_t
$
is isomorphic to
$
\widetilde{\cY} \longrightarrow \cY
$
for \(t\neq 0\), while
$\widetilde{\cD}_0 \longrightarrow (\cD_{\cX_\bullet}\cY)_0$ is isomorphic to the blowup of $\mathcal{N}_{\cX/\cY}
$ along $\cX$. Here $\cZ$ denotes the base change $\cX\times_{\cY} \widetilde{\cY}$.

The expression $\lambda_z(\bT_{\widetilde{\cD}_t})/\lambda_z(\bT_{({D_{\cX_\bullet}\cY})_t})$ is the pullback of the analogous expression from $\widetilde{\cD}$ which is determined by a class from $\cX$. Since $\cX$ and $\widetilde{\cD}$ do not depend on $t$, we may do the computation when $t=0$. That is, if $\alpha \in \Kring(\widetilde{\cX})[\![z]\!]$ satisfies equation~\eqref{eqn: def_alpha} for
$\widetilde{\cD}_0 \to (D_{\cX_\bullet}\cY)_0$, then the corresponding class in
\[
\Kring(\widetilde{\cX}\times\mathbb{A}^1)[\![z]\!] \simeq \Kring(\widetilde{\cX})[\![z]\!]
\]
satisfies equation~\ref{eqn: def_alpha} for
$\widetilde{\cD} \to D_{\cX_\bullet}\cY$. Restricting this class to
$\Kring(\widetilde{\cX})[\![z]\!]$ one obtain again $\alpha$ and sees that it satisfies
equation~\ref{eqn: def_alpha} for
\[
\widetilde{D}_t \simeq \widetilde{\cY} \longrightarrow (D_{\cX_\bullet}\cY)_t \simeq \cY \qquad (t \neq 0)
\]
We may thus assume that $\cY$ is a twisted weighted vector bundle over $\cX$. Let $\cX'\to\cX$ be the morphism in the proof of Proposition~\ref{prop: f upper shriek}. In particular, the restriction $\cY|_{\cX'}$ is a direct sum of vector bundles, each with a unique weight, and the induced map 
\[
\Kring(\cX) \longrightarrow \Kring(\cX')
\]
is injective. As in the proof of Proposition~\ref{prop: f upper shriek}, one reduces to treat the case in which $\cY$ is a weighted vector bundle obtained as pull back from the weighted normal vector bundle associated to the inclusion
\[
    B\Gm^d \subseteq [\mathbb{A}^d/\Gm^d] \longrightarrow B \Gm^d.
\]
Here $\Gm^d$ is a $d$-dimensional split torus acting on $\mathbb{A}^d$ component-wise with weight 1. Thus we reduce to the case $\cY=[\A^d/\Gm^d]$ and $\cX=B\Gm^d$, and non-necessarily distinct weights $\mathbf{w}$.
The blowup
\[
\widetilde{\cY} \simeq \Bigl[\widetilde{Y}/\Gm^d\Bigr]
\]
is the quotient of a toric stack $\widetilde{Y}$, equipped with toric boundary divisor
\[
D = \sum_i X_i + \tX,
\]
where each $X_i$ corresponds to the $i$-th coordinate hyperplane in $\mathbb{A}^d$, and $\tX$ denotes the reduced exceptional divisor. On $\widetilde{Y}$ we have a $\Gm^d$-equivariant exact sequence
\[
0 \longrightarrow \Omega_{\widetilde{Y}}^1
\longrightarrow \Omega_{\widetilde{Y}}^1(\log D)
\longrightarrow \Bigl( \bigoplus_{i} \sO_{X_i} \Bigr) \oplus \sO_{\widetilde{X}}
\longrightarrow 0.
\]

From the standard divisor exact sequence, we obtain
\[
\lambda_z^{\Gm^d}(\sO_{\widetilde{X}})
= \frac{1+z}{1+zt}
\qquad \text{and} \qquad
\lambda_z^{\Gm^d}(\sO_{X_i})
= \frac{1+z}{1+z\,t_i^{-1}t^{-w_i}} \in \Kring(\widetilde{\cY})[\![ z ]\!].
\]

Here, $\lambda_z^{\Gm^d}(-)$ of a $\Gm^d$-equivariant locally free sheaf on $Y$ denotes the $\lambda$-polynomial of the associated locally free sheaf on $\widetilde{\cY}$. Moreover, $t_i$ is the $i$-th character of $\Gm^d$, that is the pullback along the $i$-th projection $B\Gm^d \to B\mathbb{G}_m$ of the class of the standard one-dimensional representation of $\mathbb{G}_m$. The sheaf $\Omega_{\widetilde{Y}}^1(\log D)$ is $\Gm^d$-equivariantly trivial, hence
\[
\lambda_z^{\Gm^d}\!\bigl(\Omega_{\widetilde{Y}}^1(\log D)\bigr) = (1+z)^d.
\]
From the exact sequence above, we obtain
\[
\lambda_z^{\Gm^d}\!\bigl(\Omega_{\widetilde{Y}}^1\bigr)
= \frac{(1+zt)\prod_{i=1}^d (1+z\, t_i^{-1} t^{-w_i})}{1+z}.
\]
Dualizing and using Remark~\ref{rmk: tangent complex stacks}, we deduce
\[
\lambda_z\!\bigl(\bT_{\widetilde{\cY}}\bigr)
= \frac{(1+zt^{-1})\prod_{i=1}^d (1+z\, t_i t^{w_i})}{(1+z)}\cdot\frac{1}{(1+z)^d}.
\]

Similarly,
\[
\lambda_z\!\bigl(\bT_{\cY}\bigr)
= \prod_{i=1}^d (1+z\, t_i)\cdot\frac{1}{(1+z)^d}.
\]
Taking the quotient and using the fact that $[\widetilde{\cX}]=1-t$, we obtain the expression in the statement of the theorem.

Note that the map $j_*: \Kgroup(\widetilde{\cX}) \to \Kgroup(\widetilde{\cY}) $ in Corollary~\ref{cor: main exact sequence} comes with a minus sign.
\end{proof}

\bibliographystyle{amsalpha}
\bibliography{references}

\end{document}